\documentclass[a4paper,11pt]{amsart}
\usepackage[utf8]{inputenc}
\usepackage{amsmath}
\usepackage{cases}
\usepackage{amsfonts}
\usepackage[colorlinks,linkcolor=blue,citecolor=blue]{hyperref}
\usepackage{latexsym, amssymb, amsmath, amsthm, bbm}
\usepackage[all]{xy}
\usepackage{pgfplots}

\DeclareSymbolFont{EulerExtension}{U}{euex}{m}{n}
\DeclareMathSymbol{\euintop}{\mathop} {EulerExtension}{"52}
\DeclareMathSymbol{\euointop}{\mathop} {EulerExtension}{"48}

\allowdisplaybreaks[4]

\def \id{\operatorname{Id}}
\def \ker{\operatorname{Ker}}

\def \Z{\mathbb{Z}}

\def \k{\mathbbm{k}}

\def \Id{\operatorname{Id}}

\def \Aut{\operatorname{Aut}}
\def \id{\operatorname{Id}}
\def \ker{\operatorname{Ker}}

\def \Z{\mathbb{Z}}

\numberwithin{equation}{section}

\newtheorem{theorem}{Theorem}[section]
\newtheorem{lemma}[theorem]{Lemma}
\newtheorem{proposition}[theorem]{Proposition}
\newtheorem{corollary}[theorem]{Corollary}
\newtheorem{definition}[theorem]{Definition}
\newtheorem{example}[theorem]{Example}
\newtheorem{remark}[theorem]{Remark}

\begin{document}
\title{Quasitriangular structures on abelian extensions of $\mathbb{Z}_{2}$}
\thanks{$^\dag$Supported by NSFC 11722016.}


\author{Kun Zhou}
\address{Department of Mathematics, Nanjing University, Nanjing 210093, China} \email{dg1721021@smail.nju.edu.cn}
\date{}
\maketitle
\begin{abstract}  The aim of this paper is to give all quasitriangular structures on a class of semisimple Hopf algebras $\Bbbk^G\#_{\sigma,\tau}\Bbbk \mathbb{Z}_{2}$ constructed through abelian extensions of $\k\Z_2$ by $\Bbbk^G$ for an abelian group $G.$ We introduce the concept of symmetry of quasitriangular structures on Hopf algebras(see Section \ref{sec2.2} for the definition) and obtain some related propositions which can be used to simplify our calculations of quasitriangular structures. Moreover, we find that quasitriangular structures on $\Bbbk^G\#_{\sigma,\tau}\Bbbk \mathbb{Z}_{2}$ can do division-like operations. Using such operations we transform the problem of solving the quasitriangular structures into solving general solutions and giving a special solution(see Section \ref{sec5} for the definition). Then we give all general solutions for $\Bbbk^G\#_{\sigma,\tau}\Bbbk \mathbb{Z}_{2}$ and get a necessary and sufficient condition for the existence of a special solution, so we get all the quasitriangular structures on $\Bbbk^G\#_{\sigma,\tau}\Bbbk \mathbb{Z}_{2}$.
\end{abstract}

\section{Introduction}
The existence of quasitriangular structures for various families
of Hopf algebras was studied by many authors(see \cite{R5}, \cite{Ge}, \cite{N2}, \cite{N1}) in the last years. It is well known that an $\mathcal{R}$-matrix generates a solution for the quantum Yang-Baxter equation and finite dimensional quasitriangular Hopf algebras are important for
the study of knot invariants(see \cite{NR}), so it's important to know under what conditions a given Hopf algebra $H$ admits quasitriangular structures and if this is the case to determine all its $\mathcal{R}$-matrices. Such complete descriptions had been already obtained for certain families of Hopf algebras. For
instance, S. Gelaki in \cite{Ge} proved that the Hopf algebra $U_q(sl_n)'$ admits a quasitriangular structure if and only if $(n,N)=1$ or $(n,N)=2$, where $N$ is the order of $q^{1/2}$. Moreover, they had determined all possible quasitriangular structures of $U_q(sl_n)'$ up to equivalence under this condition.  In \cite{R5} D. Radford introduced a family of Hopf
algebras, $H_{N,n,q,\nu}$, which includes Taft's Hopf algebras. When these Hopf algebras are self-dual he showed that they are quasitriangular and he described all quasitriangular structures on them. After that, A. Nenciu in \cite{N1} obtained necessary and sufficient conditions for the existence of the quasitriangular structures on another family of Hopf
algebras, $H_{m,n,d,u}$, which includes the Hopf algebras $H_{N,n,q,\nu}$ and he completely determine all the quasitriangular structures of $H_{N,n,q,\nu}$ when it is quasitriangular. We note that all the Hopf algebras mentioned above are point Hopf algebras. Naturally, we ask that what about the quasitriangular structures of semisimple Hopf algebras? The simplest case of this problem is when the Hopf algebra $H=\Bbbk^G$, where $G$ is a finite group. Then all the quasitriangular structures of $H$ are given by the bicharacters of $G$. A slightly more complicated case is when $H$ is a semisimple Hopf algebra arising from exact factorizations of finite groups, such as the well-known 8-dimensional Kac-Paljutkin algebra $K_8$. The idea of constructing these semisimple Hopf algebras can be tracked back to G. Kac \cite{Kac}: Suppose that $L= G\Gamma$ is an exact factorization of the finite group $L$, into its subgroups $G$ and $\Gamma$, such that $G\cap \Gamma = 1.$ Associated to this exact factorization and appropriate cohomology data $\sigma$ and $\tau$, there is a semisimple bicrossed product Hopf algebra $H =\Bbbk^G\#_{\sigma,\tau}\Bbbk \Gamma$ (see Section 2 for the definition and \cite{M3,M5,M6} for details and generalizations). Many authors have considered quasitriangular structures of $\Bbbk^G\#_{\sigma,\tau}\Bbbk \Gamma$. For example, all possible quasitriangular structures of $K_8$ were given in \cite{S}. In 2011 S. Natale \cite{Na} proved that if $L$ is almost simple, then the extension admits no quasitriangular structure. But for our purpose, we want to find more concrete quasitriangular structures rather than absence of quasitriangular structures. So comparing the Natale's viewpoint, we consider the other extreme case which inculude the 8-dimensional Kac-Paljutkin algebra: the almost commutative case. That is, we assume that both $G$ and $\Gamma$ are commutative groups. As the start point, we further assume that $\Gamma$ is just the $\Z_2$ in this paper.

Throughout the paper we work over an algebraically closed field $\Bbbk$ of characteristic 0. In article \cite{L} we have obtained that there are only two types of quasitriangular structures on $\Bbbk^G\#_{\sigma,\tau}\Bbbk \mathbb{Z}_{2}$, one we call trivial and the other we call non-trivial. The trivial quasitriangular structures are easy to determine, so our problem is to give a necessary and sufficient condition for the existence of quasitriangular structures on $\Bbbk^G\#_{\sigma,\tau}\Bbbk \mathbb{Z}_{2}$, and give all quasitriangular structures under this condition. We first reduce the problem to determine a class of special functions, which we call quasitriangular functions. In order to get all quasitriangular functions, we analogize it to the problem of solving a system of linear equations, and it turns out that we can give all quasitriangular functions in a similar way to the solution of a system of linear equations. That is, the solution of a system of linear equations is divided into two steps, one is to find all general solutions and the other is to find a special solution. Similarly, we have two steps for general solutions and a special solution to give all quasitriangular functions. It seems to be a very complicated calculation, but we will introduce the concept of symmetry of quasitriangular structures and get some related propositions. Then we can use these propositions to simplify our calculations and so it's not as difficult as imagined.

This paper is organized as follows. In Section 2, we recall the definition of Hopf algebras $\Bbbk^G\#_{\sigma,\tau}\Bbbk \mathbb{Z}_{2}$ and give some examples of them. After that we review some main results of \cite{L} about the form of the quasitriangular structures on $\Bbbk^G\#_{\sigma,\tau}\Bbbk \mathbb{Z}_{2}$.  In Section 3, we give the concept of symmetry of quasitriangular structures on Hopf algebras and get some relevant propositions. And these propositions will be used to simplify our calculations later. In Section 4, we introduce the concept of quasitriangular functions on $\Bbbk^G\#_{\sigma,\tau}\Bbbk \mathbb{Z}_{2}$ and prove that quasitriangular structures on $\Bbbk^G\#_{\sigma,\tau}\Bbbk \mathbb{Z}_{2}$ are in one-one correspondence with quasitriangular functions of it. In Section 5, we find that a division-like operation can be done between the quasitriangular structures on $\Bbbk^G\#_{\sigma,\tau}\Bbbk \mathbb{Z}_{2}$. Using this operation, we introduce the concept of general and special solutions of quasitriangular structures on it. After that, we reduce the problem of solving all non-trivial quasitriangular structures on it to the problem of solving general and special solutions. Then we give all general solutions and we obtain a sufficient and necessary condition for the existence of special solutions on it and list all special solutions under this condition. In Section 6, we study $\varphi$-symmetric quasitriangular structures on $\Bbbk^G\#_{\sigma,\tau}\Bbbk \mathbb{Z}_{2}$ and we obtain a necessary and sufficient condition for the existence of $\varphi$-symmetric quasitriangular structures on $\Bbbk^G\#_{\sigma,\tau}\Bbbk \mathbb{Z}_{2}$. In Section 7, we apply the above conclusions to give all non-trivial quasitriangular structures on two classes Hopf algebras which we call them $K(8n,\sigma,\tau)$ and $A(8n,\sigma,\tau)$, respectively.

All Hopf algebras in this paper are finite dimensional. For the symbol $\delta$ in Section 2, we mean the classical Kronecker's symbol.

\section{Abelian extensions of $\mathbb{Z}_{2}$ and some results about it}
In this section, we recall the definition of $\Bbbk^G\#_{\sigma,\tau}\Bbbk \mathbb{Z}_{2}$, and then we give some examples of $\Bbbk^G\#_{\sigma,\tau}\Bbbk \mathbb{Z}_{2}$ for guiding our further research.
\subsection{The definition of $\Bbbk^G\#_{\sigma,\tau}\Bbbk \mathbb{Z}_{2}$}
\begin{definition}\label{def2.1.0}
A short exact sequence of Hopf algebras is a sequence of Hopf algebras
and Hopf algebra maps
\begin{equation}\label{ext}
\;\; K\xrightarrow{\iota} H \xrightarrow{\pi} A
\end{equation}
such that
\begin{itemize}
  \item[(i)] $\iota$ is injective,
  \item[(ii)]  $\pi$ is surjective,
  \item[(iii)] $\ker(\pi)= HK^+$, $K^+$ is the kernel of the counit of $K$.
\end{itemize}
\end{definition}
In this situation it is said that $H$ is an extension of $A$ by $K$ \cite[Definiton 1.4]{M3}. An extension \eqref{ext} above such that $K$ is commutative and $A$ is cocommutative is called abelian.
In this paper, we only study the following special abelian extensions
\begin{equation*}
\;\; \Bbbk^G\xrightarrow{\iota} A \xrightarrow{\pi} \Bbbk \mathbb{Z}_{2},
\end{equation*}
where $G$ is a finite abelian group. Abelian extensions were classified by Masuoka
(see \cite[Proposition 1.5]{M3}), and the above $A$ can be expressed as $\Bbbk^G\#_{\sigma,\tau}\Bbbk \mathbb{Z}_{2}$ which is defined as follows.

Let $\mathbb{Z}_2=\{1,x\}$ be the cyclic group of order 2 and let $G$ be a finite group. To give the description of $\Bbbk^G\#_{\sigma,\tau}\Bbbk \mathbb{Z}_{2}$, we need the following data
\begin{itemize}
\item[(i)] $\triangleleft :\mathbb{Z}_2 \rightarrow \Aut(G)$ is an injective group homomorphism.
\item[(ii)] $\sigma:G\rightarrow \Bbbk^\times$ is a map such that $\sigma(g\triangleleft x)=\sigma(g)$ for $g \in G$ and $\sigma(1)=1$.
\item[(iii)] $\tau:G\times G \rightarrow \Bbbk^\times$ is a unital 2-cocycle and satisfies that $\sigma(gh)\sigma(g)^{-1}\sigma(h)^{-1}=\tau(g,h)\tau(g\triangleleft x,h\triangleleft x)$\ for $g,h \in G$.
\end{itemize}
The aim of (i) is to avoid making a commutative algebra (in such case all quasitriangular structures are given by bicharacters and thus is known).
\begin{definition}\cite[Section 2.2]{AA}\label{def2.1.2}
As an algebra, the Hopf algebra $\Bbbk^G\#_{\sigma,\tau}\Bbbk \mathbb{Z}_{2}$ is generated by $\{ e_{g},x \}_{g \in G}$  satisfying
 \begin{equation*}
 e_{g}e_{h}=\delta_{g,h}e_{g},\ xe_{g}=e_{g\triangleleft x}x,\ x^2=\sum\limits_{g \in G}\sigma(g)e_{g}, \;\;\;\;g,h\in G.
 \end{equation*}
 The coproduct, counit and antipode are given by
 \begin{align*}
 &\Delta (e_{g})=\sum_{ h,k \in G,\ hk=g} e_{h}\otimes e_{k},\ \Delta(x)=[\sum\limits_{g,h \in G}\tau(g,h)e_{g}\otimes e_{h}](x\otimes x),\\
  &\epsilon(x)=1,\ \epsilon(e_{g})=\delta_{g,1}1,\\
   &\mathcal{S}(x)=\sum_{g\in G}\sigma(g)^{-1}\tau(g,g^{-1})^{-1}e_{g\triangleleft x}x,\ \mathcal{S}(e_g)=e_{g^{-1}},\;\;g\in G.
 \end{align*}
\end{definition}
The following are some examples of $\Bbbk^G\#_{\sigma,\tau}\Bbbk \mathbb{Z}_2$ and we will discuss them in next sections.
\begin{example}\label{def2.1.1}
\emph{ Let $n$ be a natural number. A Hopf algebra $H$ belonging to $\Bbbk^G\#_{\sigma,\tau}\Bbbk \mathbb{Z}_{2}$ is denoted by $K(8n,\sigma,\tau)$ if the data $(G,\triangleleft,\sigma,\tau)$ of $H$ satisfies}
  \begin{itemize}
  \item[(i)] $G=\mathbb{Z}_{2n}\times \mathbb{Z}_{2}=\langle a,b|a^{2n}=b^2=1,ab=ba\rangle;$
       \item[(ii)] $a\triangleleft x=ab,b\triangleleft x=b$.
  \end{itemize}
\end{example}
If we take $n=1$ and let $\sigma(a^ib^j)=(-1)^{(i-j)j}$ and $\tau(a^ib^j,a^kb^l)=(-1)^{j(k-l)}$ for $1\leq i,j,k,l \leq 2$, then we can easily check that the resulting 8-dimensional Hopf algebra is just the Kac-Paljutkin $8$-dimensional algebra $K_8$. Therefore, we give a kind of generalization of $K_8.$

\begin{example}\label{def2.1.3}
\emph{ Let $n$ be a natural number. A Hopf algebra $H$ belonging to $\Bbbk^G\#_{\sigma,\tau}\Bbbk \mathbb{Z}_{2}$ is denoted by $A(8n,\sigma,\tau)$ if the data $(G,\triangleleft,\sigma,\tau)$ of $H$ satisfies}
  \begin{itemize}
  \item[(i)] $G=\mathbb{Z}_{4n}=\langle a|a^{4n}=1\rangle;$
       \item[(ii)] $a\triangleleft x=a^{2n+1}$.
  \end{itemize}
\end{example}
In fact, non-trivial Hopf algebra $A(8n,\sigma,\tau)$ exists. For example we can make $\sigma(a^i)=1$ and $\tau(a^i,a^j)=(-1)^{ij}$ for $1\leq i,j \leq 4n$, then we get a non-trivial Hopf algebra $A(8n,\sigma,\tau)$.

\subsection{Some results about quasitriangular structures on $\Bbbk^G\#_{\sigma,\tau}\Bbbk \mathbb{Z}_2$}
Next, we review some results in \cite{L} about non-trivial quasitriangular structures on $\Bbbk^G\#_{\sigma,\tau}\Bbbk \mathbb{Z}_2$.

Recall that a quasitriangular Hopf algebra is a pair $(H, R)$ where $H$ is a Hopf algebra and $R=\sum R^{(1)} \otimes R^{(2)}$ is an invertible element in $H\otimes H$ such that
\begin{equation*}
 (\Delta \otimes \id)(R)=R_{13}R_{23},\; (\id \otimes \Delta)(R)=R_{13}R_{12},\;\Delta^{op}(h)R=R\Delta(h),
 \end{equation*}
for $h\in H$. Here by definition $R_{12}= \sum R^{(1)} \otimes R^{(2)}\otimes 1,\; R_{13}= \sum R^{(1)}\otimes 1 \otimes R^{(2)}$ and
$R_{23}=\sum 1 \otimes R^{(1)}\otimes R^{(2)}$. The element $R$ is called a universal $\mathcal{R}$-matrix of $H$ or a quasitriangular structure on $H$.

The first lemma is well-known.
\begin{lemma}\cite[Proposition 12.2.11]{R}\label{lem2.2.1}
Let $H$ be a Hopf algebra and $R \in H\otimes H$. For $f \in H^*$, if we denote $l(f):=(f \otimes \id)(R)$ and $r(f):=(\id \otimes f)(R)$, then the following statements are equivalent
\begin{itemize}
\item[(i)] $(\Delta \otimes \id)(R)=R_{13}R_{23}$ and $(\id \otimes \Delta)(R)=R_{13}R_{12}$.
\item[(ii)] $l(f_1)l(f_2)=l(f_1f_2)$ and $r(f_1)r(f_2)=r(f_2f_1) $ for $f_1,f_2 \in H^*$.
\end{itemize}
\end{lemma}

\begin{lemma}\cite[Lemma 3.2]{L}\label{lem2.2.2}
 Denote the dual basis of $\{e_g,e_gx\}_{g\in G}$ by $\{E_g,X_g\}_{g\in G}$, that is,
$E_g(e_h)=\delta_{g,h},\;E_g(e_hx)=0,\;X_g(e_h)=0,\;X_g(e_hx)=\delta_{g,h}$ for $g,h\in G$. Then the following equations hold in the dual Hopf algebra $(\Bbbk^G\#_{\sigma,\tau}\Bbbk \mathbb{Z}_2)^{*}$:
\begin{equation*}
E_gE_h=E_{gh},\ E_gX_h=X_hE_g=0,\ X_gX_h=\tau(g,h)X_{gh}, \;\; g,h\in G.
\end{equation*}
\end{lemma}
\begin{proof} Direct computations show that
$$E_gE_h(e_k)=E_{gh}(e_k)=\delta_{gh,k},\;\;E_gE_h(e_kx)=E_{gh}(e_kx)=0$$
 for $g,h,k\in G$. As a result, we have $E_gE_h=E_{gh}$. Similarly, one can get the last two equations.
\end{proof}

Let $\Bbbk^G\#_{\sigma,\tau}\Bbbk \mathbb{Z}_2$ as before. We need following two notions which will be used freely throughout this paper. Let
      $$S:=\{g\;|\;g\in G,\; g\triangleleft x=g\},\;\;T:=\{g\;|\; g\in G,\; g\triangleleft x\neq g\}.$$

Let $w^1:G\times G\rightarrow \Bbbk^\times$, $w^2:G\times G\rightarrow \Bbbk^\times$, $w^3:G\times G\rightarrow \Bbbk^\times$, $w^4:G\times G\rightarrow \Bbbk^\times$ be four maps (here $\Bbbk^\times=\Bbbk-\{0\}$) and define $R$ as follows
\begin{align*}
R&\colon=\sum\limits_{g,h \in G}w^1(g,h)e_{g} \otimes e_{h}+ \sum\limits_{g,h \in G}w^2(g,h)e_{g}x \otimes e_{h}+ \\
&\ \ \ \ \sum\limits_{g,h \in G}w^3(g,h)e_{g} \otimes e_{h}x+\sum\limits_{g,h \in G}w^4(g,h)e_{g}x \otimes e_{h}x.
\end{align*}
The following proposition shows that universal $\mathcal{R}$-matrices of $\Bbbk^G\#_{\sigma,\tau}\Bbbk \mathbb{Z}_{2}$ have only two possible forms.
\begin{proposition}\cite[Proposition 3.6]{L}\label{pro2.1.1}
If $R$ is a universal $\mathcal{R}$-matrix of $\Bbbk^G\#_{\sigma,\tau}\Bbbk \mathbb{Z}_{2}$, then $R$ must belong to one of the following two cases:
\begin{itemize}
             \item[(i)]$R=\sum\limits_{g,h\in G}w^1(g,h)e_g \otimes e_h$;
              \item[(ii)] $R=\sum\limits_{s_1,s_2 \in S}w^1(s_1,s_2)e_{s_1} \otimes e_{s_2}+ \sum\limits_{s \in S, t \in T}w^2(s,t)e_{s}x \otimes e_{t}+
        \sum\limits_{t \in T,s \in S}w^3(t,s)e_{t} \otimes e_{s}x+
         \sum\limits_{t_1,t_2 \in T}w^4(t_1,t_2)e_{t_1}x \otimes e_{t_2}x$.
            \end{itemize}
\end{proposition}
\begin{remark}\emph{For simple, we will call a universal $\mathcal{R}$-matrix $R$ in case (i) (resp. case (ii)) of Proposition \ref{pro2.1.1} by a \emph{trivial} (resp. \emph{non-trivial}) quasitriagular structure.}\end{remark}

To determine all non-trivial quasitriagular structures on $\Bbbk^G\#_{\sigma,\tau}\Bbbk \mathbb{Z}_{2}$, we give necessary conditions for $\Bbbk^G\#_{\sigma,\tau}\Bbbk \mathbb{Z}_{2}$ preserving a non-trivial quasitriangular structure firstly. For any finite set $X$, we use $|X|$ to denote the number of elements in $X$.
\begin{proposition}\label{pro2.1.2}
If there is a non-trivial quasitrianglar structure on $\Bbbk^G\#_{\sigma,\tau}\Bbbk \mathbb{Z}_{2}$, then
\begin{itemize}
  \item[(i)] $|S|=|T|$;
  \item[(ii)]  there is $b\in S$ such that $b^2=1$ and $t\triangleleft x=tb$ for $t\in T$;
  \item[(iii)] $\tau(s_1,s_2)=\tau(s_2,s_1),\; s_1,s_2\in S$;
\end{itemize}
\end{proposition}
\begin{proof}
(i),(ii) are part of the results in \cite[Proposition 3.8]{L}. Assume that $R$ is a non-trivial quasitrianglar structure on $\Bbbk^G\#_{\sigma,\tau}\Bbbk \mathbb{Z}_{2}$, then we have $\Delta^{op}(x)R=R\Delta(x)$. Multiply both sides of this equation by $e_{s_1}\otimes e_{s_2}$ where $s_1,s_2\in S$ and we note that $e_{s_1}\otimes e_{s_2}$ is an element in the center, so we get $(e_{s_1}\otimes e_{s_2})\Delta^{op}(x)R=R\Delta(x)(e_{s_1}\otimes e_{s_2})$.
On the one hand, we have the following equation
\begin{align*}
(e_{s_1}\otimes e_{s_2})\Delta^{op}(x)R&=(e_{s_1}\otimes e_{s_2})[\sum\limits_{g,h \in G}\tau(h,g)e_g \otimes e_h](x \otimes x)R\\
                &=[\tau(s_2,s_1)e_{s_1} \otimes e_{s_2}](x \otimes x)R\\
                &=(x \otimes x)[\tau(s_2,s_1)e_{s_1} \otimes e_{s_2}]R\\
                &=(x \otimes x)[\tau(s_2,s_1)w^1(s_1,s_2)(e_{s_1}\otimes e_{s_2}]\\
                &=\tau(s_2,s_1)w^1(s_1,s_2)e_{s_1}x\otimes e_{s_2}x
\end{align*}
On the other hand, the following equation hold
\begin{align*}
R\Delta(x)(e_{s_1}\otimes e_{s_2})&=R[\sum\limits_{g,h \in G}\tau(g,h)e_g \otimes e_h](x \otimes x)(e_{s_1}\otimes e_{s_2})\\
                &=R[\sum\limits_{g,h \in G}\tau(g,h)e_g \otimes e_h](e_{s_1}\otimes e_{s_2})(x \otimes x)\\
        &=R[\tau(s_1,s_2)e_{s_1} \otimes e_{s_2}](x \otimes x)\\
         &=[\tau(s_1,s_2)w^1(s_1,s_2)e_{s_1} \otimes e_{s_2}](x \otimes x)\\
         &=\tau(s_1,s_2)w^1(s_1,s_2)e_{s_1}x \otimes e_{s_2}x.
\end{align*}
Therefore, $(e_{s_1}\otimes e_{s_2})\Delta^{op}(x)R=R\Delta(x)(e_{s_1}\otimes e_{s_2})$ holds if and only if $\tau(s_1,s_2)=\tau(s_2,s_1)$.
\end{proof}

\begin{remark}\label{rk2.1} \emph{Since Proposition \ref{pro2.1.2} above and our aim is to find all non-trivial quasitriangular structures on $\Bbbk^G\#_{\sigma,\tau}\Bbbk \mathbb{Z}_{2}$, we agree that $\Bbbk^G\#_{\sigma,\tau}\Bbbk \mathbb{Z}_{2}$ satisfies the conditions (i)-(iii) in Proposition \ref{pro2.1.2} above in the following content.}\end{remark}

If we let
$\eta(g,h)=\tau(g,h)\tau(h,g)^{-1}$ for $g,h \in G$, then $\eta$ is a bicharacter on $G$ due to $\tau$ is a 2-cocycle on the abelian group $G$ and so (iii) of the Proposition \ref{pro2.1.2} above is equivalent to $\eta(s_1,s_2)=1$ for $s_1,s_2\in S$. We will often use $\eta$ without explaination in the following content.

\section{Symmetry of quasitriangular structures on Hopf algebras}\label{sec2.2}
We will define symmetry of quasitriangular structures on Hopf algebras and give some relevant propositions in this section. Then we apply these propositions to the special Hopf algebras $\Bbbk^G\#_{\sigma,\tau}\Bbbk \mathbb{Z}_{2}$. Let $(H,m,\eta,\Delta,\epsilon)$ be a Hopf algebra and let $R\in H\otimes H$. If $\varphi:H\rightarrow H^{op}$ is a Hopf isomorphism, then we denote $(\varphi\otimes \varphi)\circ \tau(R)$ as $R_{\varphi}$ for the sake of convenience, here $\tau$ is the flip map and $H^{op}=(H,m\circ \tau,\eta,\Delta,\epsilon)$. Now we can define $\varphi$-symmetry of quasitriangular structures on Hopf algebras as follows

\begin{definition}\label{def3.1.1}
Let $\varphi:H\rightarrow H^{op}$ be a Hopf isomorphism and let $R\in H\otimes H$, then we call $R$ is $\varphi$-symmetric if $R=R_{\varphi}$. Moreover if $R=R_{\varphi}$ and it is a quasitriangular structure on $H$ then we call $R$ is a $\varphi$-symmetric quasitriangular structure.
\end{definition}

The reason why we introduced the above definition is due to the following propositions
\begin{proposition}\label{pro3.1.1}
Let $\varphi:H\rightarrow H^{op}$ be a Hopf isomorphism and let $R\in H\otimes H$, then $R$ is a quasitriangular structure on $H$ if and only if $R_{\varphi}$ is a quasitriangular structure on $H$.
\end{proposition}

\begin{proof}
Since $R=(R_{\varphi})_{\varphi^{-1}}$, we only need to prove that if $R$ is a quasitriangular structure then $R_{\varphi}$ is a quasitriangular structure. Let $\varphi^\ast$ be the dual map of $\varphi$, then $\varphi^\ast:H^*\rightarrow H^{cop}$ is a Hopf isomorphism, here $H^{cop}=(H,m,\eta,\tau\circ\Delta,\epsilon)$. If we denote $l_R(f):=(f\otimes \Id)(R)$, $r_R(f):=(\Id \otimes f)(R)$ respectively, then we claim that the following equations hold
\begin{gather}
l_{R_\varphi}=\varphi\circ r_{R}\circ \varphi^\ast,\;
r_{R_\varphi}=\varphi\circ l_{R}\circ \varphi^\ast.\label{3.2}
\end{gather}
Directly we have
\begin{align*}
l_{R_\varphi}(f)&=(f\otimes \Id)[(\varphi\otimes \varphi)\circ \tau(R)]=(f\circ \varphi \otimes \varphi)\circ \tau(R) \\
        &=(\varphi^\ast(f) \otimes \varphi )\circ \tau(R)=(\varphi \otimes \varphi^\ast(f))(R)\\
         &=\varphi[(\Id\otimes \varphi^\ast(f))(R)]=\varphi[r_{R}\circ \varphi^\ast(f)]\\
         &=(\varphi\circ r_{R}\circ \varphi^\ast)(f),
\end{align*}
and
\begin{align*}
r_{R_\varphi}(f)&=(\Id \otimes f )[(\varphi\otimes \varphi)\circ \tau(R)]=(\varphi \otimes  f\circ\varphi )\circ \tau(R) \\
        &=(\varphi \otimes  \varphi^\ast(f)  )\circ \tau(R)=(\varphi^\ast(f)\otimes \varphi)(R)\\
         &=\varphi[(\varphi^\ast(f) \otimes \Id)(R)]=\varphi[l_{R}\circ \varphi^\ast(f)]\\
         &=(\varphi\circ l_{R}\circ \varphi^\ast)(f),
\end{align*}
So the equations \eqref{3.2} hold. Let $f_1,f_2\in H^\ast$, if we use the equations \eqref{3.2} and notice that $l_R,r_R$ are homomorphism and antihomomorphism respectively, then we have
\begin{align*}
l_{R_\varphi}(f_1f_2)&=(\varphi\circ r_{R}\circ \varphi^\ast)(f_1f_2)=(\varphi\circ r_{R})[\varphi^\ast(f_1)\varphi^\ast(f_2)]\\
         &=\varphi[r_{R}\circ\varphi^\ast(f_2)r_{R}\circ\varphi^\ast(f_1)]=(\varphi\circ r_{R}\circ \varphi^\ast)(f_1)(\varphi\circ r_{R}\circ \varphi^\ast)(f_2)
\end{align*}
and
\begin{align*}
r_{R_\varphi}(f_1f_2)&=(\varphi\circ l_{R}\circ \varphi^\ast)(f_1f_2)=(\varphi\circ l_{R})[\varphi^\ast(f_1)\varphi^\ast(f_2)]\\
         &=\varphi[l_{R}\circ\varphi^\ast(f_1)l_{R}\circ\varphi^\ast(f_2)]=(\varphi\circ l_{R}\circ \varphi^\ast)(f_2)(\varphi\circ l_{R}\circ \varphi^\ast)(f_1),
\end{align*}
therefore $(\Delta \otimes \id)(R_{\varphi})=(R_{\varphi})_{13}(R_{\varphi})_{23}$ and $(\id \otimes \Delta)(R_{\varphi})=(R_{\varphi})_{13}(R_{\varphi})_{12}$ by Lemma \ref{lem2.2.1}. To show $R_\varphi$ is a quasitriangular structure, we only need to prove $R_\varphi$ is invertible and $\Delta^{op}(h)R_\varphi=R_\varphi\Delta(h)$ for $h\in H$. Assume that $R^{-1}$ is inverse of $R$, then it can be seen that $R_\varphi (R^{-1})_\varphi=1\otimes 1$ and thus $R_\varphi$ is invertible. Suppose that $R=\Sigma_{i=1}^n r_i\otimes r^i$, and taking a $k\in H$, then we can write it by $k=\varphi(h),\; h\in H$ due to $\varphi$ is bijective map. Using $\varphi$ is Hopf isomorphism, we get
\begin{align*}
\Delta^{op}(\varphi(h))R_\varphi&=[\varphi(h_{(2)})\otimes \varphi(h_{(1)})]R_\varphi=[\varphi(h_{(2)})\otimes \varphi(h_{(1)})][\Sigma_{i=1}^n \varphi(r^i)\otimes \varphi(r_i)]\\
&=\Sigma_{i=1}^n \varphi(r^i h_{(2)})\otimes \varphi(r_ih_{(1)})=(\varphi\otimes \varphi)[\Sigma_{i=1}^n r^i h_{(2)}\otimes r_ih_{(1)}]
\end{align*}
and
\begin{align*}
R_\varphi \Delta(\varphi(h))&=R_\varphi[\varphi(h_{(1)})\otimes \varphi(h_{(2)})]=[\Sigma_{i=1}^n \varphi(r^i)\otimes \varphi(r_i)][\varphi(h_{(1)})\otimes \varphi(h_{(2)})]\\
&=\Sigma_{i=1}^n \varphi(h_{(1)} r^i)\otimes \varphi(h_{(2)} r_i)=(\varphi\otimes \varphi)[\Sigma_{i=1}^n h_{(1)} r^i\otimes h_{(2)} r_i].
\end{align*}
Because $\Delta^{op}(h)R=R\Delta(h)$, we know $(h_{(2)}\otimes h_{(1)})(\Sigma_{i=1}^n r_i\otimes r^i)=(\Sigma_{i=1}^n r_i\otimes r^i) (h_{(1)}\otimes h_{(2)})$. If we use the flip map $\tau$ acting on both sides of this equation, then we get $\Sigma_{i=1}^n r^i h_{(2)}\otimes r_ih_{(1)}=\Sigma_{i=1}^n h_{(1)} r^i\otimes h_{(2)} r_i$ and hence $\Delta^{op}(\varphi(h))R_\varphi=R_\varphi\Delta(\varphi(h))$ for $h\in H$.
\end{proof}

\begin{proposition}\label{pro3.1.2}
Let $\varphi:H\rightarrow H^{op}$ be a Hopf isomorphism and let $R\in H\otimes H$. If $R$ is $\varphi$-symmetric, then $l_R:H^\ast\rightarrow H$ is an algebra map if and only if $r_R:H^\ast\rightarrow H^{op}$ is an algebra map.
\end{proposition}

\begin{proof}
Since $R=R_\varphi$ and we have proved $l_{R_\varphi}=\varphi\circ r_{R}\circ \varphi^\ast,\;r_{R_\varphi}=\varphi\circ l_{R}\circ \varphi^\ast$ in Proposition \ref{pro3.1.1}, we get $l_{R}=\varphi\circ r_{R}\circ \varphi^\ast$ and $r_{R}=\varphi\circ l_{R}\circ \varphi^\ast$. Repeating part of the proof in Proposition \ref{pro3.1.1}, we know that $l_R$ is an algebra map if and only if $r_R:H^\ast\rightarrow H^{op}$ is an algebra map.
\end{proof}
The following corollary is an application of Proposition \ref{pro3.1.2}.
\begin{corollary}\label{coro3.1.3}
Let $\varphi:H\rightarrow H^{op}$ be a Hopf isomorphism and let $R\in H\otimes H$. If $R$ is invertible and it is $\varphi$-symmetric, then $R$ is a quasitriangular structure if and only if $l_R$ is an algebra map and $\Delta^{op}(h)R=R\Delta(h)$ for $h\in \Bbbk^G\#_{\sigma,\tau}\Bbbk \mathbb{Z}_{2}$.
\end{corollary}

\begin{proof}
By Lemma \ref{lem2.2.1} and Proposition \ref{pro3.1.2}, we get what we want.
\end{proof}

In order to apply the above results to our case, we give the following conclusions.
\begin{proposition}\label{pro3.1.3}
Let $\varphi:\Bbbk^G\#_{\sigma,\tau}\Bbbk \mathbb{Z}_{2}\rightarrow (\Bbbk^G\#_{\sigma,\tau}\Bbbk \mathbb{Z}_{2})^{op}$ be a linear map which is determined by $\varphi(e_g):=e_{g\triangleleft x},\;\varphi(e_g x):=e_{g}x$, then $\varphi$ is a Hopf isomorphism.
\end{proposition}

\begin{proof}
Obviously $\varphi$ is bijective, thus we only need to show $\varphi$ is a bialgebra map. To show $\varphi$ is an algebra map, the only non-trivial thing is to check $\varphi(e_g)\varphi(e_hx)=\varphi[(e_hx)e_g]$. Directly we have $\varphi(e_g)\varphi(e_hx)=e_{g\triangleleft x}(e_hx)=\delta_{g\triangleleft x,h}e_{h}x$ and $\varphi[(e_hx)e_g]=\varphi(e_he_{g\triangleleft x}x)=\delta_{g\triangleleft x,h}e_{h}x$, so $\varphi(e_g)\varphi(e_hx)=\varphi[(e_hx)e_g]$. To prove that $\varphi$ is a coalgebra map, we consider the dual map $\varphi^\ast$. Denote the dual basis of $\{e_g,e_gx\}_{g\in G}$ by $\{E_g,X_g\}_{g\in G}$, then it can be seen that $\varphi^\ast(E_g)=E_{g\triangleleft x}$ and $\varphi^\ast(X_g)=X_{g}$. Therefore it is easy to see that $\varphi^\ast$ is an algebra map and this implies that $\varphi$ is a coalgebra map.
\end{proof}

Let $R$ be the form (ii) in Proposition \ref{pro2.1.1} and let $\varphi$ be the Hopf isomorphism in Proposition \ref{pro3.1.3} above, then $R_{\varphi}$ is given by
\begin{gather}
              R_\varphi=\sum\limits_{s_1,s_2 \in S}w^1(s_2,s_1)e_{s_1} \otimes e_{s_2}+ \sum\limits_{s \in S, t \in T}w^3(t\triangleleft x,s)e_{s}x \otimes e_{t}+
        \sum\limits_{t \in T,s \in S}w^2(s,t\triangleleft x)e_{t} \otimes \label{r}\\
         e_{s}x+\sum\limits_{t_1,t_2 \in T}w^4(t_2,t_1)e_{t_1}x \otimes e_{t_2}x.\notag
\end{gather}

\begin{corollary} \label{coro3.1.1}
The $R$ is a quasitriangular structure on $\Bbbk^G\#_{\sigma,\tau}\Bbbk \mathbb{Z}_{2}$ if and only if $R_{\varphi}$ is a quasitriangular structure on $\Bbbk^G\#_{\sigma,\tau}\Bbbk \mathbb{Z}_{2}$.
\end{corollary}

\begin{proof}
Owing to the Proposition \ref{pro3.1.1} and Proposition \ref{pro3.1.3}, we get what we want.
\end{proof}

\begin{corollary} \label{coro3.1.2}
Let $R$ be the form (ii) in Proposition \ref{pro2.1.1}, and if $w^{i}(1\leq i\leq 4)$ satisfy the following conditions
\begin{itemize}
  \item[(i)] $w^1(s_1,s_2)=w^1(s_2,s_1)$ for $s_1,s_2\in S$;
  \item[(ii)] $w^2(s,t)=w^3(t\triangleleft x,s)$ for $s\in S,t\in T$;
   \item[(iii)] $w^4(t_1,t_2)=w^4(t_2,t_1)$ for $t_1,t_2\in T$;
\end{itemize}
then $R$ is a quasitriangular structure if and only if $l_R$ is an algebra map and $\Delta^{op}(h)R=R\Delta(h)$ for $h\in \Bbbk^G\#_{\sigma,\tau}\Bbbk \mathbb{Z}_{2}$.
\end{corollary}

\begin{proof}
It can be seen that $R=R_\varphi$ due to (i)-(iii).  Since $R_\varphi$ is given by the form \eqref{r} and (i)-(iii), we know that $R$ is $\varphi$-symmetric. Using Corollary \ref{coro3.1.3}, we get what we want.
\end{proof}

\begin{remark}\emph{For our convenience, we agree that $\varphi$ mentioned in the following content refers to the $\varphi$ in Proposition \ref{pro3.1.3}.}
\end{remark}

\section{Quasitriangular functions on $\Bbbk^G\#_{\sigma,\tau}\Bbbk \mathbb{Z}_{2}$}
In this section, we first prove that non-trivial quasitriangular structures on $\Bbbk^G\#_{\sigma,\tau}\Bbbk \mathbb{Z}_{2}$ are in one-one correspondence to some special functions on $\Bbbk^G\#_{\sigma,\tau}\Bbbk \mathbb{Z}_{2}$, which we call quasitriangular functions on $\Bbbk^G\#_{\sigma,\tau}\Bbbk \mathbb{Z}_{2}$. After that, we will focus on quasitriangular functions. Let $R$ be the form (ii) in Proposition \ref{pro2.1.1} and we will use this $R$ without explanation in the following sections, then
\begin{lemma}\label{lem4.1.1}
The equations $\Delta^{op}(h)R=R\Delta(h)$ hold for $h\in \Bbbk^G\#_{\sigma,\tau}\Bbbk \mathbb{Z}_{2}$ if and only if the following equations hold
\begin{gather}
\label{e3.12} w^2(s,t\triangleleft x)=w^2(s,t)\eta(s,t),\; s \in S,t\in T,\\
\label{e3.13} w^3(t\triangleleft x,s)=w^3(t,s)\eta(t,s),\;s \in S,t\in T,\\
\label{e3.14} \tau(t_2,t_1)w^4(t_1\triangleleft x,t_2\triangleleft x)=\tau(t_1\triangleleft x,t_2\triangleleft x)w^4(t_1,t_2),\;t_1,t_2 \in T.
\end{gather}
\end{lemma}

\begin{proof}
Since $R$ is invertible and $\Bbbk^G\#_{\sigma,\tau}\Bbbk \mathbb{Z}_{2}$ is generated by $\{e_g,x|\;g\in G\}$ as algebra, $\Delta^{op}(h)=R\Delta(h)R^{-1}$ for $h\in \Bbbk^G\#_{\sigma,\tau}\Bbbk \mathbb{Z}_{2}$ is equivalent to $\Delta^{op}(h)=R\Delta(h)R^{-1}$ for $h\in \{e_g,x|\;g\in G\}$. We first prove that $\Delta^{op}(e_g)R=R\Delta(e_g)$ for $g\in G$. Taking $s\in S,t\in T$, then directly we have
\begin{align*}
\Delta^{op}(e_s)R=[\sum\limits_{\begin{subarray}{l}  s_1,s_2 \in S  \\
                             s_1s_2=s  \\
        \end{subarray}}w^1(s_1,s_2)e_{s_1} \otimes e_{s_2}]+[\sum\limits_{\begin{subarray}{l}  t_1,t_2 \in T  \\
                             t_1t_2=s  \\
        \end{subarray}}w^4(t_1,t_2)e_{t_1}x \otimes e_{t_2}x]
\end{align*}
and
\begin{align*}
R\Delta(e_s)&=[\sum\limits_{\begin{subarray}{l}  s_1,s_2 \in S  \\
                             s_1s_2=s  \\
        \end{subarray}}w^1(s_1,s_2)e_{s_1} \otimes e_{s_2}]+[\sum\limits_{\begin{subarray}{l}  t_1,t_2 \in T  \\
                             t_1t_2=s  \\
        \end{subarray}}w^4(t_1\triangleleft x,t_2 \triangleleft x)e_{t_1\triangleleft x}x \otimes e_{t_2\triangleleft x}x].
\end{align*}
Owing to $t_1t_2=(t_1\triangleleft x)(t_2\triangleleft x)$ by definition, thus $\Delta^{op}(s)R=R\Delta(s)$.
Similarly, we have
\begin{align*}
\Delta^{op}(e_t)R=R\Delta(e_t)=[\sum\limits_{\begin{subarray}{l}  s\in S,t' \in T  \\
                             st'=s  \\
        \end{subarray}}w^2(s,t')e_{s}x \otimes e_{t'}]+[\sum\limits_{\begin{subarray}{l}  s\in S,t' \in T  \\
                             st'=s  \\
        \end{subarray}}w^3(t',s)e_{t'} \otimes e_{s}x],
\end{align*}
but $G=S\cup T$ and so we have showed $\Delta^{op}(e_g)R=R\Delta(e_g)$ for $g\in G$. Next we prove that $\Delta^{op}(x)R=R\Delta(x)$ is equivalent to above equations \eqref{e3.12}-\eqref{e3.14}. On the one hand, we have the following equation
\begin{align*}
\Delta^{op}(x)R&=[\sum\limits_{g,h \in G}\tau(h,g)e_g \otimes e_h](x \otimes x)R\\
                &=[\sum_{s_1,s_2 \in S}\tau(s_2,s_1)w^1(s_1,s_2)e_{s_1} \otimes e_{s_2}+ \sum_{s \in S, t \in T}\tau(t,s)w^2(s,t\triangleleft x)e_{s}x \otimes e_{t}+ \\
        &\sum_{t \in T, s \in S}\tau(s,t)w^3(t\triangleleft x,s)e_{t} \otimes e_{s}x+\\
        &\sum\limits_{t_1,t_2 \in T}\tau(t_2,t_1)w^4(t_1\triangleleft x,t_2 \triangleleft x)e_{t_1}x \otimes e_{t_2}x](x\otimes x),
\end{align*}
On the other hand, the following equations hold
\begin{align*}
R\Delta(x)&=R[\sum\limits_{g,h \in G}\tau(g,h)e_g \otimes e_h](x \otimes x)\\
                &=[\sum\limits_{s_1,s_2 \in S}\tau(s_1,s_2)w^1(s_1,s_2)e_{s_1} \otimes e_{s_2}+ \sum\limits_{s \in S, t \in T}\tau(s,t)w^2(s,t)e_{s}x \otimes e_{t}+\\
        &\sum\limits_{t \in T,  s \in S}\tau(t,s)w^3(t,s)e_{t} \otimes e_{s}x+\\
         &\sum\limits_{t_1,t_2 \in T}\tau(t_1\triangleleft x,t_2\triangleleft x)w^4(t_1,t_2)e_{t_1}x \otimes e_{t_2}x](x\otimes x).
\end{align*}
Therefore, $\Delta^{op}(x)R=R\Delta(x)$ holds if and only if equations \eqref{e3.12}-\eqref{e3.14} hold.
\end{proof}

If $R$ is a quasitriangular structure, then $R$ is completely determined by $w^4$. The following lemma states this fact. For simplify, we denote $\Bbbk^G\#_{\sigma,\tau}\Bbbk \mathbb{Z}_{2}$ as $H_G$ and we will use this notation without explanation in this section.

\begin{lemma}\label{lem4.1.2}
If $l(f_1)l(f_2)=l(f_1f_2)$ and $r(f_1)r(f_2)=r(f_2f_1)$ for $f_1,f_2\in (H_G)^\ast$, then $w^i(1\leq i\leq 3)$ of $R$ are completely determined by $w^4$ as follows
\begin{itemize}
  \item[(i)] $w^1(s_1,s_2)=\frac{w^4(s_1t_1,s_2t_2)w^4(t_1,t_2)}{w^4(s_1t_1,t_2)w^4(t_1,s_2t_2)}$;
  \item[(ii)] $w^2(s,t)=\tau(s,t_1)\frac{w^4(st_1,t)}{w^4(t_1,t)}$;
   \item[(iii)] $w^3(t,s)=\tau(s,t_1)\frac{w^4(t\triangleleft x,st_1)}{w^4(t\triangleleft x,t_1)}$;
\end{itemize}
where $s,s_1,s_2\in S$ and $t,t_1,t_2\in T$.
\end{lemma}

\begin{proof}
We first show (ii). Taking $s\in S,t_1\in T$, then we have $l(X_s)l(X_{t_1})=l(X_sX_{t_1})$ by our assumption. We expand this equation as follows
\begin{align*}
l(X_s)l(X_{t_1})&=[\sum\limits_{t\in T}w^2(s,t)e_{t}][\sum\limits_{t\in T}w^4(t_0,t_1)e_{t}x]=[\sum\limits_{t\in T}w^2(s,t)w^4(t_0,t)e_{t}x]
\end{align*}
and
\begin{align*}
l(X_s X_{t_1})&=\tau(s,t_1)l(X_{st_1})=[\sum\limits_{t\in T}\tau(s,t_1)w^4(st_1,t)e_{t}x],
\end{align*}
so we have $w^2(s,t)w^4(t_1,t)=\tau(s,t_1)w^4(st_1,t)$ and this implies that (ii) holds. Then we will show (i). Let $s_1,s_2\in S$ and let $t_1,t_2\in T$. Owing to $r(E_{s_2})r(E_{t_2})=r(E_{t_2}E_{s_2})$ by assumption, we can expand this equation as follows
\begin{align*}
r(E_{s_2})r(E_{t_2})&=[\sum\limits_{s_1\in S}w^1(s_1,s_2)e_{s_1}][\sum\limits_{s_1\in S}w^2(s_1,t_2)e_{s_1}x]=[\sum\limits_{s_1\in S}w^1(s_1,s_2)w^2(s_1,t_2)e_{s_1}x]
\end{align*}
and
\begin{align*}
r(E_{t_2}E_{s_2})&=r(E_{s_2t_2})=[\sum\limits_{s_1\in S}w^2(s_1,s_2t_2)e_{s_1}x],
\end{align*}
thus we get $w^1(s_1,s_2)=\frac{w^2(s_1,s_2t_2)}{w^2(s_1,t_2)}$. But we have showed the following equations
\begin{gather*}
w^2(s_1,s_2t_2)=\tau(s_1,t_1)\frac{w^4(s_1t_1,s_2t_2)}{w^4(t_1,s_2t_2)},\;w^2(s_1,t_2)=\tau(s_1,t_1)\frac{w^4(s_1t_1,t_2)}{w^4(t_1,t_2)},
\end{gather*}
therefore we know (i) holds. To show (iii), we consider $R_\varphi$(here $\varphi$ is the Hopf isomorphism in Proposition \ref{pro3.1.3}). Since the proof of Proposition \ref{pro3.1.1}, we know $R_\varphi$ satisfy $l_{R_\varphi}(f_1)l_{R_\varphi}(f_2)=l_{R_\varphi}(f_1f_2)$ and $r_{R_\varphi}(f_1)r_{R_\varphi}(f_2)=r_{R_\varphi}(f_2f_1)$ for $f_1,f_2\in (H_G)^\ast$, i.e $R_\varphi$ such that the conditions of this Lemma. Denote the $w^i(1\leq i\leq 4)$ of $R_\varphi$ by $w'^i(1\leq i\leq 4)$, then we have $w'^2(s,t)=w^3(t\triangleleft x,s)$ and $w'^4(t_1,t_2)=w^4(t_2,t_1)$ for $s\in S$ and $t_1,t_2\in T$ by \eqref{r}. But we have proved that (ii) holds, we get that $w'^2(s,t)=\tau(s,t_1)\frac{w'^4(st_1,t)}{w'^4(t_1,t)}$. And hence we know $w^3(t\triangleleft x,s)=\tau(s,t_1)\frac{w^4(t,st_1)}{w^4(t,t_1)}$ and this implies (iii).
\end{proof}
The following lemma gives a criterion for when $R$ is a non-trivial quasitriangular structure on $\Bbbk^G\#_{\sigma,\tau}\Bbbk \mathbb{Z}_{2}$.
\begin{lemma}\label{lem4.1.4}
$R$ is a quasitriangular structure on $\Bbbk^G\#_{\sigma,\tau}\Bbbk \mathbb{Z}_{2}$ if and only if the following equations hold
\begin{gather}
\label{e3.15} l(E_g)l(E_h)=l(E_gE_h),\;l(X_g)l(X_h)=l(X_gX_h),\;g,h\in G,\\
\label{e3.16} r(E_g)r(E_h)=r(E_hE_g),\;r(X_g)r(X_h)=r(X_hX_g),\;g,h\in G,\\
\label{e3.17} \tau(t_2,t_1)w^4(t_1\triangleleft x,t_2\triangleleft x)=\tau(t_1\triangleleft x,t_2\triangleleft x)w^4(t_1,t_2),\;t_1,t_2 \in T.
\end{gather}
\end{lemma}

\begin{proof}
Since the Lemma \ref{lem2.2.1}, Lemma \ref{lem4.1.1} and the definition of quasitriangular structures, we know that if $R$ is a quasitriangular structure then it satisfies the above equations \eqref{e3.15}-\eqref{e3.17}. Conversely, suppose $R$ such that equations \eqref{e3.15}-\eqref{e3.17}, we will first prove that $l(f_1)l(f_2)=l(f_1f_2)$ and $r(f_1)r(f_2)=r(f_2f_1)$ for $f_1,f_2\in (\Bbbk^G\#_{\sigma,\tau}\Bbbk \mathbb{Z}_{2})^\ast$. For $s\in S,t\in T$, since
\begin{align*}
l(E_s)=\sum\limits_{s'\in S}w^1(s,s')e_{s'},\;l(E_t)=\sum\limits_{s'\in S}w^3(t,s')e_{s'}x
\end{align*}
and
\begin{align*}
l(X_s)=\sum\limits_{t'\in T}w^2(s,t')e_{t'},\;l(X_t)=\sum\limits_{t'\in T}w^4(t,t')e_{t'}x,
\end{align*}
it can be seen that $l(E_g)l(X_h)=l(E_gX_h)=0$ and $l(X_g)l(E_h)=l(X_gE_h)=0$ for $g,h\in G$. Because $(\Bbbk^G\#_{\sigma,\tau}\Bbbk \mathbb{Z}_{2})^\ast$ is linear spanned by $\{E_g,X_g|\;g\in G\}$, we get that $l(f_1)l(f_2)=l(f_1f_2)$ for $f_1,f_2\in (\Bbbk^G\#_{\sigma,\tau}\Bbbk \mathbb{Z}_{2})^\ast$. Using a similar discussion, we will know that $r(f_1)r(f_2)=r(f_2f_1)$ for $f_1,f_2\in (\Bbbk^G\#_{\sigma,\tau}\Bbbk \mathbb{Z}_{2})^\ast$. Secondly, we will prove that $\Delta^{op}(h)R=R\Delta(h)$ for $h\in \Bbbk^G\#_{\sigma,\tau}\Bbbk \mathbb{Z}_{2}$. Owing to the Lemma \ref{lem4.1.1}, we only need to prove that $w^2(s,t\triangleleft x)=w^2(s,t)\eta(s,t)$ and $w^3(t\triangleleft x,s)=w^3(t,s)\eta(t,s)$ for $s\in S,t\in T$. Let $t_0\in T$, then we have the following equations by (ii) of Lemma \ref{lem4.1.2}
\begin{gather}
w^2(s,t)=\tau(s,t_0)\frac{w^4(st_0,t)}{w^4(t_0,t)},\;w^2(s,t\triangleleft x)=\tau(s,t_0\triangleleft x)\frac{w^4(st_0\triangleleft x,t\triangleleft x)}{w^4(t_0\triangleleft x,t\triangleleft x)}.\label{e4.1}
\end{gather}
Due to the assumption, we have
\begin{align}
w^4(st_0\triangleleft x,t\triangleleft x)=\frac{\tau(st_0\triangleleft x,t\triangleleft x)}{\tau(t,st_0)}w^4(st_0,t)\label{e4.2}
\end{align}
and
\begin{align}
w^4(t_0\triangleleft x,t\triangleleft x)=\frac{\tau(t_0\triangleleft x,t\triangleleft x)}{\tau(t,t_0)}w^4(t_0,t).\label{e4.3}
\end{align}
So $w^2(s,t\triangleleft x)=w^2(s,t) \frac{\tau(s,t_0\triangleleft x)}{\tau(s,t_0)} \frac{\tau(st_0\triangleleft x,t\triangleleft x)}{\tau(t,st_0)} \frac{\tau(t,t_0)}{\tau(t_0\triangleleft x,t\triangleleft x)}$. Using $\tau$ is two cocycle, we get
\begin{align*}
\frac{\tau(s,t_0\triangleleft x)}{\tau(s,t_0)} \frac{\tau(st_0\triangleleft x,t\triangleleft x)}{\tau(t,st_0)} \frac{\tau(t,t_0)}{\tau(t_0\triangleleft x,t\triangleleft x)}&=\frac{\tau(s,t_0t) \tau(t_0\triangleleft x,t\triangleleft x)}{\tau(s,t_0)\tau(t,st_0)}\frac{\tau(t,t_0)}{\tau(t_0\triangleleft x,t\triangleleft x)}\\
                &=\frac{\tau(s,t_0t) \tau(t_0\triangleleft x,t\triangleleft x)}{\eta(s,t_0)\tau(t_0,s)\tau(t,st_0)}\frac{\tau(t,t_0)}{\tau(t_0\triangleleft x,t\triangleleft x)}\\
        &=\frac{\tau(s,t_0t) \tau(t_0\triangleleft x,t\triangleleft x)}{\eta(s,t_0)\tau(tt_0,s)\tau(t,t_0)}\frac{\tau(t,t_0)}{\tau(t_0\triangleleft x,t\triangleleft x)}\\
        &=\frac{\eta(s,tt_0)}{\eta(s,t_0)}=\eta(s,t).
\end{align*}
Therefore $w^2(s,t\triangleleft x)=w^2(s,t)\eta(s,t)$. To show $w^3(t\triangleleft x,s)=w^3(t,s)\eta(t,s)$, we consider $R_\varphi$ and denote the $w^i(1\leq i\leq 4)$ of $R_\varphi$ by $w'^i(1\leq i\leq 4)$, then we have $w'^2(s,t)=w^3(t\triangleleft x,s)$ and $w'^4(t_1,t_2)=w^4(t_2,t_1)$ for $s\in S,t_1,t_2\in T$ by \eqref{r}. Owing to $\tau(t_1\triangleleft x,t_2\triangleleft x)\tau(t_1,t_2)=\tau(t_2\triangleleft x,t_1\triangleleft x)\tau(t_2,t_1)=\sigma(t_1t_2)\sigma(t_1)^{-1}\sigma(t_2)^{-1}$, then we have $\frac{\tau(t_1\triangleleft x,t_2\triangleleft x)}{\tau(t_2,t_1)}=\frac{\tau(t_2\triangleleft x,t_1\triangleleft x)}{\tau(t_1,t_2)}$ and hence it is easy to see that $R_\varphi$ also satisfies the conditions of this Lemma. But we have showed that $w'^2(s,t\triangleleft x)=w'^2(s,t)\eta(s,t)$, so $w^3(t,s)=w^3(t\triangleleft x,s)\eta(s,t)$. Since $\eta(s,t)^{-1}=\eta(t,s)$, we know $w^3(t\triangleleft x,s)=w^3(t,s)\eta(t,s)$ and therefore we have completed the proof.
\end{proof}

Since $H_G$ is determined by the data $(G,\triangleleft,\sigma,\tau)$, naturally we can guess that all non-trivial quasitriagular structures on $H_G$ can be expressed by using the data $(G,\triangleleft,\sigma,\tau)$. To confirm this conjecture, we use the following propositions.

\begin{proposition}\label{pro4.1.1}
If $R$ is a universal $\mathcal{R}$-matrix of $H_G$, then
\begin{itemize}
             \item[(i)]$\tau(s,t_1)\frac{w^4(st_1,t)}{w^4(t_1,t)}=\tau(s,t_2)\frac{w^4(st_2,t)}{w^4(t_2,t)}$;
              \item[(ii)]$\tau(s,t_1)\frac{w^4(t,st_1)}{w^4(t,t_1)}=\tau(s,t_2)\frac{w^4(t,st_2)}{w^4(t,t_2)}$;
              \item[(iii)]$w^4(t,t_1)w^4(t^{-1},t_1\triangleleft x)\sigma(t_1)=\tau(t,t^{-1})$;
              \item[(iv)]$w^4(t_1,t)w^4(t_1\triangleleft x,t^{-1})\sigma(t_1)=\tau(t,t^{-1})$;
              \item[(v)]$w^4(t_1\triangleleft x,t_2\triangleleft x)=\frac{\tau(t_1\triangleleft x,t_2\triangleleft x)}{\tau(t_2,t_1)}w^4(t_1,t_2)$;
            \end{itemize}
where $s\in S,\;t,t_1,t_2\in T$.
\end{proposition}

\begin{proof}
Since (ii) of Lemma \ref{lem4.1.2}, we know $w^2(s,t)=\tau(s,t_1)\frac{w^4(st_1,t)}{w^4(t_1,t)}=\tau(s,t_2)\frac{w^4(st_2,t)}{w^4(t_2,t)}$ for $s\in S$ and $t,t_1\in T$. Therefore (i) holds. Similarly, we get (ii) due to (iii) of Lemma \ref{lem4.1.2}. Owing to $R$ is a universal $\mathcal{R}$-matrix, we have $l(X_t)l(X_{t^{-1}})=l(X_t X_{t^{-1}})$. Then we expand the equation as follows
\begin{align*}
l(X_t)l(X_{t^{-1}})&=[\sum\limits_{t_1\in T}w^4(t,t_1)e_{t_1}x][\sum\limits_{t_1\in T}w^4(t^{-1},t_1)e_{t_1}x]\\
&=[\sum\limits_{t_1\in T}w^4(t,t_1)w^4(t,t_1\triangleleft x)\sigma(t_1)e_{t_1}]
\end{align*}
and
\begin{align*}
l(X_t X_{t^{-1}})&=\tau(t,t^{-1})l(X_{1})=\tau(t,t^{-1})[\sum\limits_{t\in T}w^2(1,t_1)e_{t_1}x]=[\sum\limits_{t\in T}\tau(t,t^{-1})e_{t_1}x],
\end{align*}
so we have (iii). Similarly, we get $r(X_t)r(X_{t^{-1}})=r(X_{t^{-1}}X_t )$. And if we expand this equation then we get $w^4(t_1,t)w^4(t_1\triangleleft x,t^{-1})\sigma(t_1)=\tau(t^{-1},t)$. But $\tau(t^{-1},t)=\tau(t,t^{-1})$ due to $\eta(t,t^{-1})=1$, therefore we know (iv) holds. (v) is a conclusion of Lemma \ref{lem4.1.4} and so we have completed the proof.
\end{proof}

In fact, given a function $w:T\times T\rightarrow \Bbbk^\times$ that satisfies the above conditions, we can find a unique quasitriangular structure $R$ that satisfies $w^4=w$. And we will prove this conclusion in Theorem \ref{thm4.1.1}. Because of this reason, we introduce the concept of quasitriangular functions on $H_G$.

\begin{definition}\label{def4.1.1}
A \emph{quasitriangular function} on $H_G$ is a function $w:T\times T\rightarrow \Bbbk^\times$ such that (i)-(v) in Proposition \ref{pro4.1.1}, i.e it satisfies the following condtions
\begin{itemize}
             \item[(i)]$\tau(s,t_1)\frac{w(st_1,t)}{w(t_1,t)}=\tau(s,t_2)\frac{w(st_2,t)}{w(t_2,t)}$;
              \item[(ii)]$\tau(s,t_1)\frac{w(t,st_1)}{w(t,t_1)}=\tau(s,t_2)\frac{w(t,st_2)}{w(t,t_2)}$;
              \item[(iii)]$w(t,t_1)w(t^{-1},t_1\triangleleft x)\sigma(t_1)=\tau(t,t^{-1})$;
              \item[(iv)]$w(t_1,t)w(t_1\triangleleft x,t^{-1})\sigma(t_1)=\tau(t,t^{-1})$;
              \item[(v)]$w(t_1\triangleleft x,t_2\triangleleft x)=\frac{\tau(t_1\triangleleft x,t_2\triangleleft x)}{\tau(t_2,t_1)}w(t_1,t_2)$;
            \end{itemize}
where $s\in S,\;t,t_1,t_2\in T$.
\end{definition}
It can be seen that the definition of a quasitriangular function is expressed by the data $(G,\triangleleft,\sigma,\tau)$. Furthermore, we will see that non-trivial quasitriangular structures on $H_G$ are in one-one correspondence to quasitriangular functions on it in Corollary \ref{coro4.1.1}. Since we will often deal with the two maps $l_R,r_R$ in the later sections, we give the following lemmas about them

\begin{lemma}\label{lem4.1.a}
Let $R$ be the form (ii) in Proposition \ref{pro2.1.1}, then we have
\begin{itemize}
  \item[(i)] $l(E_{s_1})l(E_{s_2})=l(E_{s_1s_2})\Leftrightarrow w^1(s_1s_2,s)=w^1(s_1,s)w^1(s_2,s),\;s\in S$;
  \item[(ii)] $l(E_s)l(E_t)=l(E_{st})\Leftrightarrow w^1(s,s')w^3(t,s')=w^3(st,s'),\;s'\in S$;
   \item[(iii)] $l(X_{s_1})l(X_{s_2})=l(X_{s_1}X_{s_2})\Leftrightarrow w^2(s_1,t)w^2(s_2,t)=\tau(s_1,s_2)w^2(s_1s_2,t),\;t\in T$;
   \item[(iv)] $l(X_{s})l(X_{t})=l(X_{s}X_{t})\Leftrightarrow w^2(s,t')w^4(t,t')=\tau(s,t)w^4(st,t'),\;t'\in T$;
\end{itemize}
where $s,s_1,s_2\in S$ and $t,t_1,t_2\in T$.
\end{lemma}

\begin{proof}
We only show (i) and the other things can be proved in a similar way. Since
\begin{align*}
l(E_{s_1})l(E_{s_2})&=[\sum\limits_{s\in S}w^1(s_1,s)e_{s}][\sum\limits_{s\in S}w^1(s_2,s)e_{s}]=\sum\limits_{s\in S}w^1(s_1,s)w^1(s_2,s)e_{s}
\end{align*}
and
\begin{align*}
l(E_{s_1s_2})&=\sum\limits_{s\in S}w^1(s_1s_2,s)e_{s},
\end{align*}
we know (i) holds.
\end{proof}

\begin{lemma}\label{lem4.1.b}
Let $R$ be the form (ii) in Proposition \ref{pro2.1.1}, then we have
\begin{itemize}
  \item[(i)] $r(E_{s_1})r(E_{s_2})=r(E_{s_1s_2})\Leftrightarrow w^1(s,s_1s_2)=w^1(s,s_1)w^1(s,s_2),\;s\in S$;
  \item[(ii)] $r(E_s)r(E_t)=r(E_{st})\Leftrightarrow w^1(s',s)w^2(s',t)=w^2(s',st),\;s'\in S$;
   \item[(iii)] $r(X_{s_1})r(X_{s_2})=r(X_{s_2}X_{s_1})\Leftrightarrow w^3(t,s_1)w^3(t,s_2)=\tau(s_2,s_1)w^3(t,s_1s_2),\;t\in T$;
   \item[(iv)] $r(X_{t})r(X_{s})=r(X_{s}X_{t})\Leftrightarrow w^3(t'\triangleleft x,s)w^4(t',t)=\tau(s,t)w^4(t',st),\;t'\in T$;
\end{itemize}
where $s,s_1,s_2\in S$ and $t,t_1,t_2\in T$.
\end{lemma}

\begin{proof}
Similar to the proof of Lemma \ref{lem4.1.a} above.
\end{proof}

\begin{lemma}\label{lem4.1.5}
Let $R$ be the form (ii) in Proposition \ref{pro2.1.1}, and if $w^4$ is a quasitriangular function on $H_G$ and $w^i(1\leq i\leq 3)$ are given in Lemma \ref{lem4.1.2}, then $l(X_g)l(X_h)=l(X_gX_h),\;g,h\in G$.
\end{lemma}

\begin{proof}
Since $w^2(s,t)=\tau(s,t_1)\frac{w^4(st_1,t)}{w^4(t_1,t)}$ for $t_1\in T$ by assumption and (iv) of the Lemma \ref{lem4.1.a}, we have $l(X_s)l(X_{t_1})=l(X_sX_{t_1})$. Similarly, if we repeat part of the proof in Proposition \ref{pro4.1.1}, then we will get that $l(X_t)l(X_{t^{-1}})=l(X_t X_{t^{-1}})$ is equivalent to $w^4(t,t_1)w^4(t^{-1},t_1\triangleleft x)\sigma(t_1)=\tau(t,t^{-1})$ for $t_1\in T$. But we have assumed that $w^4(t,t_1)w^4(t^{-1},t_1\triangleleft x)\sigma(t_1)=\tau(t,t^{-1})$ for $t_1\in T$, therefore we have $l(X_t)l(X_{t^{-1}})=l(X_t X_{t^{-1}})$. To show $l(X_g)l(X_h)=l(X_gX_h)$ for $g,h\in G$, we only need to show the following equations hold
\begin{gather*} l(X_{t_1})l(X_{t_2})=l(X_{t_1}X_{t_2}),\;l(X_{t_1})l(X_s)=l(X_{t_1}X_s),\;l(X_{s_1})l(X_{s_2})=l(X_{s_1}X_{s_2}),
\end{gather*}
where $s,s_1,s_2\in S$ and $t_1,t_2\in T$. Since $|S|=|T|$ and $T=T^{-1}$, where $T^{-1}:=\{t^{-1}|\;t\in T\}$, we can assume $t_1=st$ and $t_2=t^{-1}$, then we have
\begin{align*}
l(X_{t_1})l(X_{t_2})&=l(X_{st})l(X_{t^{-1}})=[\tau(s,t)^{-1}l(X_s)l(X_{t})]l(X_{t^{-1}})\\
&=\tau(s,t)^{-1}l(X_s)[l(X_{t})l(X_{t^{-1}})]=\tau(s,t)^{-1}l(X_s)l(X_{t}X_{t^{-1}})\\
&=\tau(s,t)^{-1}l(X_s)[\tau(t,t^{-1})l(X_1)]=\tau(s,t)^{-1}\tau(t,t^{-1})l(X_s).
\end{align*}
It can be seen that $X_{st}X_{t^{-1}}=\tau(s,t)^{-1}\tau(t,t^{-1})X_s$ by using the $\tau$ is a 2-cocycle, and hence $l(X_{t_1})l(X_{t_2})=l(X_{t_1}X_{t_2})$. For $s\in S$, we can find $t,t'$ such that $s=tt'$ due to $|S|=|T|$. Because $t_1t\in S$ by definition, we have
\begin{align*}
l(X_{t_1})l(X_{s})&=l(X_{t_1})l(X_{tt'})=l(X_{t_1})[\tau(t,t')^{-1}l(X_{t})l(X_{t'})]\\
                    &=\tau(t,t')^{-1}[l(X_{t_1})l(X_{t})]l(X_{t'})=\tau(t,t')^{-1}l(X_{t_1}X_{t})l(X_{t'})\\
                    &=\tau(t,t')^{-1}\tau(t_1,t)l(X_{t_1t})l(X_{t'})=\tau(t,t')^{-1}\tau(t_1,t)l(X_{t_1t}X_{t'}).
\end{align*}
Similarly, one can show $X_{t_1}X_{tt'}=\tau(t,t')^{-1}\tau(t_1,t)X_{t_1t}X_{t'}$ and hence $ l(X_{t_1})l(X_s)=l(X_{t_1}X_s)$. To show $l(X_{s_1})l(X_{s_2})=l(X_{s_1}X_{s_2})$ for $s_1,s_2\in S$, we assume that $s_2=tt'$ for some $t,t'\in T$. Then we have
\begin{align*}
l(X_{s_1})l(X_{s_2})&=l(X_{s_1})l(X_{tt'})=l(X_{s_1})[\tau(t,t')^{-1}l(X_{t})l(X_{t'})]\\
                    &=\tau(t,t')^{-1}[l(X_{s_1})l(X_{t})]l(X_{t'})=\tau(t,t')^{-1}l(X_{s_1}X_{t})l(X_{t'})\\
                    &=\tau(t,t')^{-1}\tau(s_1,t)l(X_{s_1t})l(X_{t'})=\tau(t,t')^{-1}\tau(s_1,t)l(X_{s_1t}X_{t'}).
\end{align*}
One can check that $X_{s_1}X_{s_2}=\tau(t,t')^{-1}\tau(s_1,t)X_{s_1t}X_{t'}$, so $l(X_{s_1})l(X_{s_2})=l(X_{s_1}X_{s_2})$. Therefore we have completed the proof.
\end{proof}

\begin{lemma}\label{lem4.1.6}
Let $R$ be in Lemma \ref{lem4.1.5}, then we have $l(E_g)l(E_h)=l(E_gE_h)$ for $g,h\in G$.
\end{lemma}

\begin{proof}
We mimic the proof of above Lemma \ref{lem4.1.5}. Let $s\in S,t\in T$, then we have
\begin{align}
\label{Est1} l(E_s)l(E_t)&=[\sum\limits_{s'\in S}w^1(s,s')e_{s'}][\sum\limits_{s'\in S}w^3(t,s')e_{s'}x]=[\sum\limits_{s'\in S}w^1(s,s')w^3(t,s')e_{s'}x]
\end{align}
and
\begin{align}
\label{Est2} l(E_sE_t)&=l(E_{st})=[\sum\limits_{s'\in S}w^3(st,s')e_{s'}x].
\end{align}
therefore we need to show $w^1(s,s')w^3(t,s')=w^3(st,s')$ for $s'\in S$ if we want to prove $l(E_s)l(E_t)=l(E_sE_t)$. Let $t_1:=t\triangleleft x$ and taking $t_2\in T$, since we have assumed $w^3$ such that (iii) of Lemma \ref{lem4.1.2}, we have
\begin{align*}
w^3(t,s')=\tau(s',t_2)\frac{w^4(t_1,s't_2)}{w^4(t_1,t_2)},\;w^3(st,s')=\tau(s',t_2)\frac{w^4(st_1,s't_2)}{w^4(st_1,t_2)}.
\end{align*}
And hence $\frac{w^3(st,s')}{w^3(t,s')}=\frac{w^4(st_1,s't_2)w^4(t_1,t_2)}{w^4(st_1,t_2)w^4(t_1,s't_2)}$. Because $w^1$ satisfy the (i) of Lemma \ref{lem4.1.2}, we know $\frac{w^3(st,s')}{w^3(t,s')}=w^1(s,s')$ and thus we get $l(E_s)l(E_t)=l(E_sE_t)$. Then we prove that $l(E_t)l(E_{t^{-1}})=l(E_1)$. Since
\begin{align*}
l(E_t)l(E_{t^{-1}})&=[\sum\limits_{s'\in S}w^3(t,s')e_{s'}x][\sum\limits_{s'\in S}w^3(t^{-1},s')e_{s'}x]\\
&=[\sum\limits_{s'\in S}w^3(t,s')w^3(t^{-1},s')\sigma(s')e_{s'}]
\end{align*}
and
\begin{align*}
w^3(t,s')w^3(t^{-1},s')&=[\tau(s',t_1)\frac{w^4(t\triangleleft x,s't_1)}{w^4(t\triangleleft x,t_1)}][\tau(s',t_1\triangleleft x)\frac{w^4(t^{-1}\triangleleft x,s't_1\triangleleft x)}{w^4(t^{-1}\triangleleft x,t_1\triangleleft x)}]\\
&=\tau(s',t_1)\tau(s',t_1\triangleleft x)\frac{w^4(t\triangleleft x,s't_1)w^4(t^{-1}\triangleleft x,s't_1\triangleleft x)}{w^4(t\triangleleft x,t_1)w^4(t^{-1}\triangleleft x,t_1\triangleleft x)}\\
&=\tau(s',t_1)\tau(s',t_1\triangleleft x)\frac{\tau(t\triangleleft x,t^{-1}\triangleleft x)\sigma(t_1\triangleleft x)}{\tau(t\triangleleft x,t^{-1}\triangleleft x)\sigma(s't_1\triangleleft x)}\\
&=\tau(s',t_1)\tau(s',t_1\triangleleft x)\sigma(t_1\triangleleft x)\frac{1}{\sigma(s't_1\triangleleft x)}=\sigma(s')^{-1}.
\end{align*}
The first equality follows from the assumption about $w^3$, and the third one follows from
Relation (iii) in Proposition \ref{pro4.1.1} and the last one follows from the compatibility of $\sigma$ and $\tau$. Thus we have showed $w^3(t,s')w^3(t^{-1},s')\sigma(s')=1$, and this implies $l(E_t)l(E_{t^{-1}})=l(E_1)$. Because we have proved that $l(E_s)l(E_t)=l(E_sE_t)$ and $l(E_t)l(E_{t^{-1}})=l(E_1)$ for $s\in S,t\in T$, if we repeat the proof of Lemma \ref{lem4.1.5} then we can obtain $l(E_g)l(E_h)=l(E_gE_h),\;g,h\in G$.
\end{proof}

\begin{lemma}\label{lem4.1.7}
Let $R$ be in Lemma \ref{lem4.1.5}, then $r_R$ is an algebra anti-homomorphism.
\end{lemma}

\begin{proof}
If we consider $R_\varphi$ then it is easy to see that $R_\varphi$ also satisfies the conditions of Lemma \ref{lem4.1.5}, so we can apply Lemma \ref{lem4.1.5}-\ref{lem4.1.6} to $R_\varphi$, i.e we know $l_{R_\varphi}$ is an algebra map. Since we have showed $l_{R_\varphi}=\varphi\circ r_{R}\circ \varphi^\ast$, then $l_{R_\varphi}$ is an algebra map implies $r_R$ is antihomomorphism.
\end{proof}
Now we prove the inverse of Proposition \ref{pro4.1.1}.
\begin{theorem}\label{thm4.1.1}
Assume $w$ is a quasitriangular function on $H_G$, then there is a unique $R$ such that it is a non-trivial quasitriangular structure on $H_G$ and the $w^4$ of it is equal to the $w$.
\end{theorem}

\begin{proof}
Uniqueness can be obtained directly from Lemma \ref{lem4.1.2}. To show the existence, we will use the $w$ to construct a non-trivial quasitriangular structure. We define $w^i(1\leq i\leq 4)$ of $R$ through letting $w^4:=w$ and let $w^i(1\leq i\leq 3)$ be given by (i)-(iii) of Lemma \ref{lem4.1.2}. Since $w$ is a quasitriangular function, we know $w^2$ and $w^3$ are well defined. By direct calculation we can get $w^1(s_1,s_2)=\frac{w^2(s_1,s_2t_2)}{w^2(s_1,t_2)}$ for $t_2\in T$, so $w^1$ is also well-defined. Owing to Lemma \ref{lem4.1.5}-\ref{lem4.1.7}, we know $R$ such that the equations \eqref{e3.15}, \eqref{e3.16} in Lemma \ref{lem4.1.4}. Furthermore, the $R$ satisfies the equation \eqref{e3.17} of Lemma \ref{lem4.1.4} by the definition of quasitriangular function, so $R$ is a non-trivial quasitriangular structure on $H_G$ due to Lemma \ref{lem4.1.4}.
\end{proof}

\begin{corollary}\label{coro4.1.1}
There is a bijective map between the set of non-trivial quasitriangular structures on $H_G$ and the set of quasitriangular functions on $H_G$.
\end{corollary}

\begin{proof}
Denote the set of non-trivial quasitriangular structures on $H_G$ as $N$ and we write the set of quasitriangular functions on $H_G$ as $F$, then we can define a map $\phi:N\rightarrow F$ by $\phi(R):=w^4$. Since Proposition \ref{pro4.1.1}, we know $\phi$ is well defined. Owing to Theorem \ref{thm4.1.1}, we get $\phi$ is bijective.
\end{proof}

Given a function $w:T\times T\rightarrow \Bbbk^\times$ and taking $t_0\in T$, then we can define functions $w^2:S\times T\rightarrow \Bbbk^\times$ and $w^3:T\times S\rightarrow \Bbbk^\times$ as follows
\begin{align}
\label{w23} w^2(s,t):=\tau(s,t_0)\frac{w^4(st_0,t)}{w^4(t_0,t)},\;w^3(t,s):=\tau(s,t_0)\frac{w^4(t\triangleleft x,st_0)}{w^4(t\triangleleft x,t_0)},\;s\in S,t\in T.
\end{align}
Let $V_G$ be the subspace of $(H_G)^\ast$ which is linear spanned by $\{X_g|\;g\in G\}$, then we can define $l_w:V_G\rightarrow H_G$ and $r_w:V_G\rightarrow H_G$ through letting
\begin{align}
\label{lw} l_w(X_s):=\sum\limits_{t'\in T}w^2(s,t')e_{t'},\; l_w(X_t):=\sum\limits_{t'\in T}w(t,t')e_{t'}x,
\end{align}
\begin{align}
\label{rw} r_w(X_s):=\sum\limits_{t'\in T}w^3(t',s)e_{t'},\; r_w(X_t):=\sum\limits_{t'\in T}w(t',t)e_{t'}x.
\end{align}
It can be seen that $V_G$ is a subalgebra of $(H_G)^\ast$. In order to determine when the function $w$ is a quasitriangular function on $(H_G)^\ast$, we give the following propositions

\begin{proposition}\label{pro4.1.2}
$w$ such that (i)-(iv) in Definition \ref{def4.1.1} if and only if $l_w$ is an algebra homomorphism and $r_w$ is an algebra anti-homomorphism.
\end{proposition}

\begin{proof}
If $w$ such that (i)-(iv) in Definition \ref{def4.1.1} then we can repeat the proof of Lemma \ref{lem4.1.5}, and hence we know $l_w$ is an algebra homomorphism and $r_w$ is an algebra antihomomorphism. On the contrary, if $l_w$ is an algebra homomorphism and $r_w$ is an algebra antihomomorphism then we have the following equations
\begin{align*}
l(X_s) l(X_{t_1})&=l(X_s X_{t_1}),\;l(X_t) l(X_{t^{-1}})=l(X_tX_{t^{-1}}),
\end{align*}
\begin{align*}
r(X_s) r(X_{t_1})&=r(X_{t_1}X_s),\;r(X_t) r(X_{t^{-1}})=r(X_{t^{-1}}X_t).
\end{align*}
Expand these equations above, then we know that $w$ such that (i)-(iv) in Definition \ref{def4.1.1}.
\end{proof}

Proposition \ref{pro4.1.2} above will often be used to solve quasitriangular functions on $H_G$ in the following sections. Given a function $w:T\times T\rightarrow \Bbbk^\times$ and taking $t_0\in T$, we have defined $w^2,w^3$ through the equalities \eqref{w23}. Furthermore, we can define another function $w^1:S\times S\rightarrow \Bbbk^\times$ as follows
\begin{align}
\label{w1} w^1(s_1,s_2)=\frac{w^4(s_1t_0,s_2t_0)w^4(t_0,t_0)}{w^4(s_1t_0,t_0)w^4(t_0,s_2t_0)},\;s_1,s_2\in S.
\end{align}
Then we have the following proposition

\begin{proposition}\label{pro4.1.3}
Taking $t_0\in T$ and if $w$ such that (i)-(iv) in Definition \ref{def4.1.1}, then $w$ is a quasitriangular function on $H_G$ if and only if the following equations hold
\begin{itemize}
             \item[(i)]$w^1(s,b)=w^1(b,s)=\eta(t_0,s)$,\;here $w^1$ is given by the \eqref{w1} above;
              \item[(ii)]$w(t_0\triangleleft x,t_0\triangleleft x)=\frac{\tau(t_0\triangleleft x,t_0\triangleleft x)}{\tau(t_0,t_0)}w(t_0,t_0)$;
            \end{itemize}
\end{proposition}

\begin{proof}
If $w$ is a quasitriangular function on $H_G$, then we only need to show (i). Since $w$ is a quasitriangular function on $H_G$, we can find a unique $R\in H_G\otimes H_G$ such that $R$ is a non-trivial quasitriangular structure on $H_G$ and the $w^4$ of it is equal to the $w$ by Theorem \ref{thm4.1.1}. Since Lemma \ref{lem4.1.2}, we know the $w^i(1\leq i \leq 3)$ of $R$ are given by the equations \eqref{w23}, \eqref{w1}. Owing to Lemma \ref{lem4.1.1}, we have
\begin{align}
\label{e4.4} w^2(s,t\triangleleft x)=w^2(s,t)\eta(s,t),\;w^3(t\triangleleft x,s)=w^3(t,s)\eta(t,s),
\end{align}
where $s\in S,t\in T$. Owing to (ii) of Lemma \ref{lem4.1.a}, we get $w^3(bt,s)=w^1(b,s)w^3(t,s)$. But $bt=t\triangleleft x$ because of the Remark \ref{rk2.1}, we know $w^3(t\triangleleft x,s)=w^1(b,s)w^3(t,s)$. Since $w^3(t\triangleleft x,s)=w^3(t,s)\eta(t,s)$ by \eqref{e4.4}, we get $w^1(b,s)=\eta(t,s)$. Due to $\eta$ is a bicharacter and the Remark \ref{rk2.1}, we know $\eta(t,s)=\eta(t_0,s)$ and hence $w^1(b,s)=\eta(t_0,s)$. Similarly, we can show $w^1(s,b)=\eta(t_0,s)$ and thus we have shown (i). Conversely, if $w$ satisfies (i), (ii), then we can construct a $R\in H_G\otimes H_G$ such that $w^4=w$ and the $w^i(1\leq i \leq 3)$ of it are given by the equations \eqref{w23}, \eqref{w1}. To show $w$ is a quasitriangular function, we need only to prove that $w(t_1\triangleleft x,t_2\triangleleft x)=\frac{\tau(t_1\triangleleft x,t_2\triangleleft x)}{\tau(t_2,t_1)}w(t_1,t_2)$ for $t_1,t_2\in T$. Repeating the proofs of Lemmas \ref{lem4.1.5}-\ref{lem4.1.7}, then we know $l_R$ is an algebra homomorphism and $r_R$ is an algebra anti-homomorphism. So we have $l_R(E_{b})l_R(E_{t})=l_R(E_{bt})$ and $r_R(E_{b})r_R(E_{t})=r_R(E_{bt})$ for $t\in T$. But we have already seen that these two equalities implies that $w^3(t\triangleleft x,s)=w^1(b,s)w^3(t,s)$ and $w^2(s,t\triangleleft x)=w^1(s,b)w^2(s,t)$. Because of (i), we get
\begin{align}
\label{e4.5} w^2(s,t\triangleleft x)=w^2(s,t)\eta(s,t),\;w^3(t\triangleleft x,s)=w^3(t,s)\eta(t,s).
\end{align}
Since $w^2,w^3$ are given by the equations \eqref{w23}, one can get
\begin{align}
\label{e4.6} w^4(s_1t_0,s_2t_0)=\tau(s_1,t_0)^{-1}w^2(s_1,s_2t_0)w^4(t_0,s_2t_0)
\end{align}
and
\begin{align}
\label{e4.7} w^4(t_0,s_2t_0)=\tau(s_2,t_0)^{-1}w^4(t_0,t_0)w^3(t_0\triangleleft x,s_2),
\end{align}
where $s_1,s_2\in S$. Using equations \eqref{e4.6} and \eqref{e4.7} together, then we get
\begin{align}
\label{e4.8} w^4(s_1t_0,s_2t_0)=\frac{w^2(s_1,s_2t_0)w^3(t_0\triangleleft x,s_2)w^4(t_0,t_0)}{\tau(s_1,t_0)\tau(s_2,t_0)}.
\end{align}
Similarly, one can get
\begin{align}
\label{e4.9} w^4(s_1t_0\triangleleft x,s_2t_0\triangleleft x)=\frac{w^2(s_1,s_2t_0\triangleleft x)w^3(t_0,s_2)w^4(t_0 \triangleleft x ,t_0 \triangleleft x)}{\tau(s_1,t_0\triangleleft x)\tau(s_2,t_0\triangleleft x)}.
\end{align}
Combining the equations \eqref{e4.5}, \eqref{e4.8}, \eqref{e4.9}, we obtain
\begin{align*}
\frac{w^4(s_1t_0\triangleleft x,s_2t_0\triangleleft x)}{w^4(s_1t_0,s_2t_0)}&=\eta(s_1,s_2t_0) \frac{1}{\eta(t_0,s_2)} \frac{w^4(t_0 \triangleleft x ,t_0 \triangleleft x)}{w^4(t_0,t_0)} \frac{\tau(s_1,t_0)\tau(s_2,t_0)}{\tau(s_1,t_0\triangleleft x)\tau(s_2,t_0\triangleleft x)}\\
&=\eta(s_1,s_2t_0) \frac{1}{\eta(t_0,s_2)} \frac{\tau(t_0 \triangleleft x ,t_0 \triangleleft x)}{\tau(t_0,t_0)} \frac{\tau(s_1,t_0)\tau(s_2,t_0)}{\tau(s_1,t_0\triangleleft x)\tau(s_2,t_0\triangleleft x)}.
\end{align*}
Using the following Lemma \ref{lem4.1.8}, we obtain $\frac{w^4(s_1t_0\triangleleft x,s_2t_0\triangleleft x)}{w^4(s_1t_0,s_2t_0)}=\frac{\tau(s_1t_0\triangleleft x,s_2t_0\triangleleft x)}{\tau(s_2t_0,s_1t_0)}$. Since $T=t_0S$, we know $w(t_1\triangleleft x,t_2\triangleleft x)=\frac{\tau(t_1\triangleleft x,t_2\triangleleft x)}{\tau(t_2,t_1)}w(t_1,t_2)$ for $t_1,t_2\in T$.
\end{proof}

Proposition \ref{pro4.1.3} above simplifies the test for the condition (v) in Definition \ref{def4.1.1}, so it will be used frequently in next sections.

\begin{lemma}\label{lem4.1.8}
$\eta(s_1,s_2t_0) \frac{1}{\eta(t_0,s_2)} \frac{\tau(t_0 \triangleleft x ,t_0 \triangleleft x)}{\tau(t_0,t_0)} \frac{\tau(s_1,t_0)\tau(s_2,t_0)}{\tau(s_1,t_0\triangleleft x)\tau(s_2,t_0\triangleleft x)}=\frac{\tau(s_1t_0\triangleleft x,s_2t_0\triangleleft x)}{\tau(s_2t_0,s_1t_0)}$.
\end{lemma}

\begin{proof}
Directly we have
\begin{align*}
X_{s_1}X_{s_2}X_{t_0}X_{t_0}&=X_{s_1}(X_{s_2}X_{t_0})X_{t_0}\\
&=\eta(s_2,t_0)X_{s_1}(X_{t_0}X_{s_2})X_{t_0}\\
&=\eta(s_2,t_0)(X_{s_1}X_{t_0})(X_{s_2}X_{t_0})\\
&=\eta(s_2,t_0)[\tau(s_1,t_0)X_{s_1t_0}][\tau(s_2,t_0)X_{s_2t_0}]\\
&=\eta(s_2,t_0)\tau(s_1,t_0)\tau(s_2,t_0)X_{s_1t_0}X_{s_2t_0}\\
&=\eta(s_2,t_0)\tau(s_1,t_0)\tau(s_2,t_0)\tau(s_1t_0,s_2t_0)X_{s_1s_2t_0^2}
\end{align*}
and
\begin{align*}
X_{s_1}X_{s_2}X_{t_0}X_{t_0}&=X_{s_1}X_{s_2}(X_{t_0}X_{t_0})\\
&=X_{s_1}X_{s_2}[\frac{\tau(t_0,t_0)}{\tau(t_0\triangleleft x,t_0 \triangleleft x)}X_{t_0\triangleleft x}X_{t_0 \triangleleft x}]\\
&=\frac{\tau(t_0,t_0)}{\tau(t_0\triangleleft x,t_0 \triangleleft x)}X_{s_1}(X_{s_2}X_{t_0 \triangleleft x})X_{t_0 \triangleleft x}\\
&=\frac{\tau(t_0,t_0)}{\tau(t_0\triangleleft x,t_0 \triangleleft x)}X_{s_1}[\eta(s_2,t_0 \triangleleft x)X_{t_0 \triangleleft x}X_{s_2}]X_{t_0 \triangleleft x}\\
&=\eta(s_2,t_0 \triangleleft x)\frac{\tau(t_0,t_0)}{\tau(t_0\triangleleft x,t_0 \triangleleft x)}(X_{s_1}X_{t_0 \triangleleft x})(X_{s_2}X_{t_0 \triangleleft x})\\
&=\eta(s_2,t_0 \triangleleft x)\frac{\tau(t_0,t_0)}{\tau(t_0\triangleleft x,t_0 \triangleleft x)} \tau(s_1,t_0 \triangleleft x)\tau(s_2,t_0 \triangleleft x)X_{s_1t_0 \triangleleft x}X_{s_2t_0 \triangleleft x}.
\end{align*}
Because $X_{s_1t_0 \triangleleft x}X_{s_2t_0 \triangleleft x}=\tau(s_1t_0 \triangleleft x,s_2t_0 \triangleleft x)X_{s_1s_2t_0^2}$, we know
\begin{align*}
\frac{\eta(s_1t_0,s_2t_0)}{\eta(s_2,t_0 \triangleleft x)} \eta(s_2,t_0) \frac{\tau(t_0 \triangleleft x ,t_0 \triangleleft x)}{\tau(t_0,t_0)} \frac{\tau(s_1,t_0)\tau(s_2,t_0)}{\tau(s_1,t_0\triangleleft x)\tau(s_2,t_0\triangleleft x)}=\frac{\tau(s_1t_0\triangleleft x,s_2t_0\triangleleft x)}{\tau(s_2t_0,s_1t_0)}.
\end{align*}
To complete the proof, we only need to show $\eta(s_2,t_0)=\eta(t_0,s_2)^{-1}$ and
$\frac{\eta(s_1t_0,s_2t_0)}{\eta(s_2,t_0 \triangleleft x)}=\eta(s_1,s_2t_0)$. By the definition of $\eta$, we have $\eta(s_2,t_0)=\eta(t_0,s_2)^{-1}$. Since
\begin{align*}
\frac{\eta(s_1t_0,s_2t_0)}{\eta(s_2,t_0 \triangleleft x)}&=\frac{\eta(s_1,s_2t_0)\eta(t_0,s_2t_0)}{\eta(s_2,t_0 \triangleleft x)}=\frac{\eta(s_1,s_2t_0)\eta(t_0,s_2)}{\eta(s_2,t_0 \triangleleft x)}\\
&=\eta(s_1,s_2t_0)\eta(t_0,s_2)\eta(t_0\triangleleft x,s_2)=\eta(s_1,s_2t_0)\eta(t_0t_0\triangleleft x,s_2)\\
&=\eta(s_1,s_2t_0),
\end{align*}
the last equality follows from the assumption about $H_G$ in Remark \ref{rk2.1}, we know $\frac{\eta(s_1t_0,s_2t_0)}{\eta(s_2,t_0 \triangleleft x)}=\eta(s_1,s_2t_0)$.
\end{proof}

\section{Solutions of quasitriangular structures on $\Bbbk^G\#_{\sigma,\tau}\Bbbk \mathbb{Z}_{2}$}\label{sec5}
\subsection{General solutions for quasitriangular structures on $\Bbbk^G\#_{\sigma,\tau}\Bbbk \mathbb{Z}_{2}$}
Let $R,R'$ be non-trivial quastriangular structures on $\Bbbk^G\#_{\sigma,\tau}\Bbbk \mathbb{Z}_{2}$ and assume that the four maps associated with $R$ (resp. $R'$) are $w^i(1\leq i\leq 4)$ (resp. $w'^i(1\leq i\leq 4)$), then we can use these maps to define four other maps $v^i(1\leq i\leq 4)$ as follows
\begin{align*}
v^1(s_1,s_2):=\frac{w^1(s_1,s_2)}{w'^1(s_1,s_2)},\;v^2(s,t):=\frac{w^2(s,t)}{w'^2(s,t)}\\
v^3(t,s):=\frac{w^3(t,s)}{w'^3(t,s)},\;v^4(t_1,t_2):=\frac{w^4(t_1,t_2)}{w'^4(t_1,t_2)},
\end{align*}
where $s,s_1,s_2\in S$ and $t,t_1,t_2\in T$. Using the data $(G,\triangleleft,\sigma,\tau)$ of $\Bbbk^G\#_{\sigma,\tau}\Bbbk \mathbb{Z}_{2}$ we can induce another data $(G',\triangleleft',\sigma',\tau')$ by making $G':=G,\;\triangleleft':=\triangleleft$ and $\sigma'(g):= 1,\;\tau'(g,h):=1$ for $g,h\in G$. Then the data $(G',\triangleleft',\sigma',\tau')$ determines a Hopf algebra by Definition \ref{def2.1.2} and we simply denote it as $\Bbbk^G\#\Bbbk \mathbb{Z}_{2}$. Then we have

\begin{proposition}\label{pro5.1.1}
Let $R''$ be the form (ii) in Proposition \ref{pro2.1.1} and the $(w'')^i(1\leq i\leq 4)$ of it are the $v^i(1\leq i\leq 4)$ above, then $R''$ is a quasitriangular structure on $\Bbbk^G\#\Bbbk \mathbb{Z}_{2}$.
\end{proposition}

\begin{proof}
Since $R,R'$ are non-trivial quasitriangular structures on $\Bbbk^G\#_{\sigma,\tau}\Bbbk \mathbb{Z}_{2}$, we know $w^4,\;w'^4$ are quasitriangular functions on $\Bbbk^G\#_{\sigma,\tau}\Bbbk \mathbb{Z}_{2}$. Then we can easily check that $v^4$ is a quasitriangular function on $\Bbbk^G\#\Bbbk \mathbb{Z}_{2}$. And furthermore, $v^i(1\leq i\leq 3)$ are given as follows because of Lemma \ref{lem4.1.2}
\begin{align*}
v^1(s_1,s_2)=\frac{v^4(s_1t_1,s_2t_2)v^4(t_1,t_2)}{v^4(s_1t_1,t_2)v^4(t_1,s_2t_2)},\;v^2(s,t)=\frac{v^4(st_1,t)}{v^4(t_1,t)},\;v^3(t,s)=\frac{v^4(t\triangleleft x,st_1)}{v^4(t\triangleleft x,t_1)}.
\end{align*}
Therefore we know $R''$ is a quasitriangular structure on $\Bbbk^G\#\Bbbk \mathbb{Z}_{2}$ due to the proof of Theorem \ref{thm4.1.1}.
\end{proof}

We can view $R''$ as $\frac{R}{R'}$, then we can analogize the solutions of a system of linear equations and give the following definition
\begin{definition}\label{def5.1.1}
We call a \emph{quasitriangular structure} on $\Bbbk^G\#\Bbbk \mathbb{Z}_{2}$ as a general solution for $\Bbbk^G\#_{\sigma,\tau}\Bbbk \mathbb{Z}_{2}$. Naturally, we call a quasitriangular structure on $\Bbbk^G\#_{\sigma,\tau}\Bbbk \mathbb{Z}_{2}$ as a \emph{special solution for $\Bbbk^G\#_{\sigma,\tau}\Bbbk \mathbb{Z}_{2}$}.
\end{definition}
Then the problem of solving all quasitriangular structures on $\Bbbk^G\#_{\sigma,\tau}\Bbbk \mathbb{Z}_{2}$ can be reduced to solve all general solutions and find a special solution. In this subsection, we will give all general solutions for $\Bbbk^G\#_{\sigma,\tau}\Bbbk \mathbb{Z}_{2}$. To do this, we first give the following lemma

\begin{lemma}\label{lem5.1.1}
Assume that $H$ is a finite abelian group and $\phi:H\rightarrow H$ is a group isomorphism. Let $S_H:=\{h\in H|\;\phi(h)=h\}$ and let $T_H:=\{h\in H|\;\phi(h)\neq h\}$. If $|S_H|=|T_H|$ and there is $c\in H$ such that $c^2=1$ and $\phi(h)=hc$ for $h\in T_H$, then there are $s_1,...,s_n\in S_H$ and $a\in T_H$ such that $H=\langle s_i,a|\;s_i^{k_i}=1,a^2=s_1^{m_1}...s_n^{m_n}, s_is_j=s_js_i,as_i=s_ia\rangle_{1\leq i,j\leq n}$ as group for some natural numbers $n,k_i,m_j$.
\end{lemma}

\begin{proof}
Since $S_H$ is a subgroup of $H$, we can find $s_1...s_m\in S_H$ such that $S_H= \langle s_i|\;s_i^{k_i}=1, s_is_j=s_js_i\rangle_{1\leq i,j\leq n}$ for some natural numbers $k_i(1\leq i \leq n)$. Because $T_H$ is not empty, we can find $a\in T_H$. But $a^2\in S_H$ due to $\phi(a^2)=(ac)^2=a^2$, so we can assume $a^2=s_1^{m_1}...s_n^{m_n}$ for some natural numbers $m_i(1\leq i \leq n)$. Let $H'$ be a universal group such that $H'=\langle S_i,A|\;S_i^{k_i}=1,A^2=S_1^{m_1}...S_n^{m_n}, S_iS_j=S_jS_i,AS_i=S_iA\rangle_{1\leq i,j\leq n}$, then we will prove that $H\cong H'$ as group and hence we completed the proof. We define a group homomorphism $f:H'\rightarrow H$ through letting $f(S_i):=s_i,f(A):=a$ for $1\leq i \leq n$, then $f$ is well defined by the definition of $H'$. Owing to $aS_H\subseteq T_H$ and $|S_H|=|T_H|$, we obtain $T_H=aS_H$. Thus we can see that $f$ is surjective. To show $f$ is injective, we only need to show $|H'|\leq |H|$. Let $S_{H'}:=\langle s_i\rangle_{1\leq i \leq n}$, then $f|_{S_{H'}}:S_{H'}\rightarrow S_H$ is onto by definition and hence $|S_{H'}|\geq |S_{H}|$. But $|S_{H'}|\leq |S_{H}|$ by definition of $H'$, so we know $|S_{H'}|=|S_{H}|$. Directly, we have $H'=S_{H'}\cup AS_{H'}$ and therefore $|H'|\leq 2|S_{H'}|$. Since $|H|=2|S_{H}|$ and $|S_{H'}|=|S_{H}|$, we have $|H'|\leq |H|$.
\end{proof}

\begin{corollary} \label{coro5.1.1}
For the data $(G,\triangleleft,\sigma,\tau)$, there are $s_1,...,s_n\in S$ and $a\in T$ such that $G = \langle s_i,a|\;s_i^{k_i}=1,a^2=s_1^{m_1}...s_n^{m_n}, s_is_j=s_js_i,as_i=s_ia\rangle_{1\leq i,j\leq n}$ as group for some natural numbers $n,k_i,m_j$.
\end{corollary}

\begin{proof}
Since the assumption about $\Bbbk^G\#_{\sigma,\tau}\Bbbk \mathbb{Z}_{2}$ in Remark \ref{rk2.1} and the Lemma \ref{lem5.1.1} above, we get what we want.
\end{proof}

By Corollary \ref{coro5.1.1}, we can assume that $G=\langle s_i,a|\;s_i^{k_i}=1,a^2=s_1^{m_1}...s_n^{m_n}, s_is_j=s_js_i,as_i=s_ia\rangle_{1\leq i,j\leq n}$ as group for some natural numbers $n,k_i,m_j$ in the following content, where $s_i\in S,a\in T$ for $1\leq i \leq n$. Let $\Bbbk^G\#_{\sigma,\tau}\Bbbk \mathbb{Z}_{2}$ as before, we associate a free object with it as follows. We define $F_G$ as a
free $\Bbbk$ algebra generated by set $\{x_{s_i},x_a,x_1,e_{s_i},e_a,e_1\}_{1 \leq i \leq n}$, and let $I_G$ be the ideal generated by $\{x_{s_i}x_1-x_{s_i},x_1x_{s_i}-x_{s_i},x_{s_i}x_{s_j}-x_{s_j}x_{s_i},x_{s_i}x_{a}-\eta(s_i,a)x_{a}x_{s_i},x_{s_i}^{k_i}-P_{s_i^{k_i}}x_1,x_a^2-\tau(a,a)P_{s_1^{m_1}...s_n^{m_n}}^{-1}x_{s_1}^{m_1}...x_{s_n}^{m_n}\}$
$\cup \{x_1e_1,e_1x_1,e_{s_i}e_1-e_{s_i},e_1e_{s_i}-e_{s_i},e_{s_i}e_{s_j}-e_{s_j}e_{s_i},e_{s_i}e_{a}-e_{a}e_{s_i},e_{s_i}^{k_i}-e_1,e_a^2-e_{s_1}^{m_1}...e_{s_n}^{m_n}\}$, where $1 \leq i,j \leq n$ and $P_{g_1^{j_1}...g_n^{j_n}}\in \Bbbk$ is defined by the following equation
\begin{align}
\label{e5.1.p} X_{g_1}^{j_1}...X_{g_n}^{j_n}=P_{g_1^{j_1}...g_n^{j_n}}X_{g_1^{j_1}...g_n^{j_n}},\;g_1,...,g_n\in G,j_1,...,j_n\in \mathbb{N}.
\end{align}

For convenience, we agree that $P_{g_1^{j_1}...g_n^{j_n}}$ in the following content refers to the $P_{g_1^{j_1}...g_n^{j_n}}$ in equation \eqref{e5.1.p} above. Then we have the following lemma.

\begin{lemma}\label{lemm5.1.c}
Denote the dual Hopf algebra of $\Bbbk^G\#_{\sigma,\tau}\Bbbk \mathbb{Z}_{2}$ by $H^*$, then $H^*\cong A_G/I_G$ as an algebra.
\end{lemma}

\begin{proof}
We define an algebra map $\pi:A_G/I_G\rightarrow H^*$ by setting
\begin{gather*}
\pi(x_{s_i})=X_{s_i},\; \pi(x_a)=X_a,\;\pi(x_1)=X_1,\; \pi(e_{s_i})=E_{s_i},\; \pi(e_a)=E_a,\;\pi(e_1)=E_1.
\end{gather*}
Then we will show that $\pi$ is well defined and it is bijective. Since Lemma
\ref{lem2.2.2} and the definition of $I_G$, we know $\pi$ is well defined. Next we show $\pi$ is bijective. Because $\{X_{s_i},X_{a},E_{s_i},E_{a}\}\subseteq \textrm{Im}\pi$, we know $\pi$ is surjective. Owing to the definition of $I_G$, we obtain  $\text{dim}(A_G/I_G)\leq \text{dim}(H^*)$. But we have shown $\pi$ is surjective, so $\text{dim}(A_G/I_G)=\text{dim}(H^*)$ and hence $\pi$ is bijective.
\end{proof}

Let $R$ be the form (ii) on $\Bbbk^G\#_{\sigma,\tau}\Bbbk \mathbb{Z}_{2}$ in Proposition \ref{pro2.1.1}. For our purposes, we assume that $w^i(1\leq i\leq 4)$ of $R$ satisfy $w^1(1,s)=w^1(s,1)=1$ and $w^2(1,t)=w^3(t,1)$ for $s\in S,t\in T$ in the following content. Then we have
\begin{lemma}\label{lem5.1.2}
The map $l_R$ is an algebra homomorphism if and only if the following conditions hold
\begin{itemize}
  \item[(i)] $l(E_{s_1})^{i_1}...l(E_{s_n})^{i_n}=l(E_{s_1^{i_1}...s_n^{i_n}}),\;l(X_{s_1})^{i_1}...l(X_{s_n})^{i_n}=P_{s_1^{i_1}...s_n^{i_n}}l(X_{s_1^{i_1}...s_n^{i_n}})$;
  \item[(ii)] $l(E_{s})l(E_a)=l(E_{sa}),\;l(X_{s})l(X_a)=\tau(s,a)l(X_{sa})$;
   \item[(iii)]$l(E_{s_i})^{k_i}=l(E_1),\;l(X_{s_i})^{k_i}=P_{{s_i}^{k_i}}l(X_1)$;
   \item[(iv)] $l(E_a)^2=l(E_{s_1^{m_1}...s_1^{m_n}}),\;l(X_a)^2=P_{a^2}\;l(X_{s_1^{m_1}...s_1^{m_n}})$;
   \item[(v)] $w^2(s,t\triangleleft x)=\eta(s,t)w^2(s,t)$;
\end{itemize}
where $s\in S,\;1\leq i \leq n$.
\end{lemma}

\begin{proof}
We define an algebra map $\pi:H^*\rightarrow H$ by setting
\begin{gather*}
\pi(X_{s_i})=l(X_{s_i}),\; \pi(X_a)=l(X_a),\;\pi(X_1)=l(X_1),\\ \pi(E_{s_i})=l(E_{s_i}),\; \pi(E_a)=l(E_a),\;\pi(E_1)=l(E_1).
\end{gather*}
Then we will show that $\pi$ is well defined and $\pi=l_R$. To show $\pi$ is well defined, the only non-trivial case is to prove that $\pi(X_s)\pi(X_a)=\eta(s,a)l(X_a)l(X_s)$. Directly we have $\pi(X_s)\pi(X_a)=l(X_s)l(X_a)$ and
\begin{align*}
l(X_{s})l(X_a)=[\sum\limits_{t\in T}w^2(s,t)e_{t}][\sum\limits_{t\in T}w^4(a,t)e_{t}x]=[\sum\limits_{t\in T}w^2(s,t)w^4(a,t)e_{t}x].
\end{align*}
Similarly, we get $\pi(X_a)\pi(X_s)=l(X_a)l(X_s)$ and
\begin{align*}
l(X_a)l(X_{s})=[\sum\limits_{t\in T}w^4(a,t)e_{t}x][\sum\limits_{t\in T}w^2(s,t)e_{t}]=[\sum\limits_{t\in T}w^2(s,t\triangleleft x)w^4(a,t)e_{t}x].
\end{align*}
Owing to (v), we obtain $\pi(X_s)\pi(X_a)=\eta(s,a)l(X_a)l(X_s)$. Due to (i),(ii), we get $\pi=l_R$ and thus we have completed the proof.
\end{proof}
Similar to Lemma \ref{lem5.1.2}, we have
\begin{lemma}\label{lem5.1.3}
The map $r_R$ is an algebra anti-homomorphism if and only if the following conditions hold
\begin{itemize}
  \item[(i)] $r(E_{s_1})^{i_1}...r(E_{s_n})^{i_n}=r(E_{s_1^{i_1}...s_n^{i_n}}),\;r(X_{s_1})^{i_1}...r(X_{s_n})^{i_n}=P_{s_1^{i_1}...s_n^{i_n}}r(X_{s_1^{i_1}...s_n^{i_n}})$;
  \item[(ii)] $r(E_{s})r(E_a)=r(E_{sa}),\;r(X_a)r(X_{s})=\tau(s,a)r(X_{sa})$;
   \item[(iii)]$r(E_{s_i})^{k_i}=r(E_1),\;r(X_{s_i})^{k_i}=P_{{s_i}^{k_i}}\;r(X_1),\;r(X_a)^2=P_{a^2}\;r(X_{s_1^{m_1}...s_1^{m_n}})$;
   \item[(iv)] $r(E_a)^2=r(E_{s_1^{m_1}...s_1^{m_n}}),\;r(X_a)^2=P_{a^2}\;r(X_{s_1^{m_1}...s_1^{m_n}})$;
   \item[(v)] $w^3(t\triangleleft x,s)=\eta(t,s)w^3(t,s)$;
\end{itemize}
where $s\in S,\;1\leq i \leq n$.
\end{lemma}

\begin{proof}
Consider the $R_\varphi$, then it can be seen that $R_\varphi$ such that the conditions of Lemma \ref{lem5.1.2} and so $l_{R_\varphi}$ is an algebra map. Since the proof of Proposition \ref{pro3.1.1}, we know $l_{R_\varphi}$ is an algebra map if and only if $r_R$ is an algebra anti-homomorphism and hence we have completed the proof.
\end{proof}

The following proposition can be used to determine when $R$ is a quasitriangular structure on $\Bbbk^G\#_{\sigma,\tau}\Bbbk \mathbb{Z}_{2}$.
\begin{proposition}\label{pro5.1.2}
Let $R$ be the form (ii) on $\Bbbk^G\#_{\sigma,\tau}\Bbbk \mathbb{Z}_{2}$ in Proposition \ref{pro2.1.1}, then $R$ is a quasitriangular structure on $\Bbbk^G\#_{\sigma,\tau}\Bbbk \mathbb{Z}_{2}$ if and only if $R$ such that the conditions of Lemma \ref{lem5.1.2}, Lemma \ref{lem5.1.3} and $w^4(ab,ab)=\frac{\tau(ab,ab)}{\tau(a,a)}w^4(a,a)$.
\end{proposition}

\begin{proof}
If $R$ is a quasitriangular structure on $\Bbbk^G\#_{\sigma,\tau}\Bbbk \mathbb{Z}_{2}$, Since $l_R$ is an algebra homomorphism and $r_R$ is an algebra anti-homomorphism and thus $R$ such that the conditions of Lemma \ref{lem5.1.2}, Lemma \ref{lem5.1.3}. Owing to Lemma \ref{lem4.1.1}, we get $w^4(ab,ab)=\frac{\tau(ab,ab)}{\tau(a,a)}w^4(a,a)$. Conversely, if $R$ such that the conditions of Lemma \ref{lem5.1.2}, Lemma \ref{lem5.1.3} and $w^4(ab,ab)=\frac{\tau(ab,ab)}{\tau(a,a)}w^4(a,a)$, then $l_R$ is an algebra homomorphism and $r_R$ is an algebra anti-homomorphism. Therefore $w^4$ satisfies (i)-(iv) in Definition \ref{def4.1.1} due to Proposition \ref{pro4.1.2}. Then we will use Proposition \ref{pro4.1.3} to get what we want. To do this, we only need to show $w^1(s,b)=w^1(b,s)=\eta(a,s)$. Because $l_R$ is an algebra map and (ii) of Lemma \ref{lem4.1.b}, we get $w^2(s,ab)=w^2(s,a)w^1(s,b)$. Owing to Lemma \ref{lem5.1.2}, we obtain $w^2(s,a\triangleleft x)=w^2(s,a)\eta(s,a)$. But $a\triangleleft x=ab$ by assumption, so we have $w^1(s,b)=\eta(s,a)$. Due to $\frac{\eta(s,a)}{\eta(a,s)}=\eta(s,a^2)$ and $a^2\in S$, we know $\eta(s,a)=\eta(a,s)$ and hence $w^1(s,b)=\eta(a,s)$. Similarly, one can show $w^1(b,s)=\eta(a,s)$ and so we have completed the proof.
\end{proof}

In practice, we usually use the following corollary to determine when $R$ is a quasitriangular structure on $\Bbbk^G\#_{\sigma,\tau}\Bbbk \mathbb{Z}_{2}$ because it's easier to be checked.
\begin{corollary}\label{coro5.1.2}
Let $R$ be the form (ii) on $\Bbbk^G\#_{\sigma,\tau}\Bbbk \mathbb{Z}_{2}$ in Proposition \ref{pro2.1.1}, then $R$ is a quasitriangular structure on $\Bbbk^G\#_{\sigma,\tau}\Bbbk \mathbb{Z}_{2}$ if and only if $R$ such that the following conditions
\begin{itemize}
  \item[(i)] $R$ satisfies (i)-(iv) of Lemma \ref{lem5.1.2};
  \item[(ii)] $R$ satisfies (i)-(iv) of Lemma \ref{lem5.1.3};
   \item[(iii)]$w^1(b,s)=w^1(s,b)=\eta(a,s),\;s\in S$;
   \item[(iv)] $w^4(ab,ab)=\frac{\tau(ab,ab)}{\tau(a,a)}w^4(a,a)$;
\end{itemize}
\end{corollary}

\begin{proof}
Since the proof of Proposition \ref{pro5.1.2}, we know that necessity holds. Conversely, owing to Proposition \ref{pro5.1.2} above, we only need to show $w^2(s,t\triangleleft x)=\eta(s,t)w^2(s,t)$ and $w^3(t\triangleleft x,s)=\eta(t,s)w^3(t,s)$ for $s\in S,t\in T$. Since (i),(ii) of Lemma \ref{lem5.1.3}, we know $R$ satisfies (ii) of Lemma \ref{lem4.1.b}. Then we get $w^2(s,tb)=w^2(s,t)w^1(s,b)$. But we have $w^1(s,b)=\eta(a,s)$, so $w^2(s,tb)=w^2(s,t)\eta(a,s)$ by Lemma \ref{lem4.1.b}. Due to $\frac{\eta(a,s)}{\eta(s,t)}=\eta(at,s)$ and $at\in S$, we get $\frac{\eta(a,s)}{\eta(s,t)}=1$ and hence $w^2(s,t\triangleleft x)=\eta(s,t)w^2(s,t)$. Similarly, one can prove that $w^3(t\triangleleft x,s)=\eta(t,s)w^3(t,s)$ and thus we have completed the proof.
\end{proof}

For the convenience of calculation, we assume $b=s_1^{p_1}...s_n^{p_n}$ for some natural numbers $p_1,...,p_n$ in the following sections. Then we use the following propositions to give all general solutions.
\begin{proposition}\label{pro5.1.3}
Let $R$ be a general solution for $\Bbbk^G\#_{\sigma,\tau}\Bbbk \mathbb{Z}_{2}$, that is to say $R$ is a non-trivial quasitriangular structure on $\Bbbk^G\#\Bbbk \mathbb{Z}_{2}$, and if we denote
\begin{align}
\label{e5.x} \alpha_{ij}:=w^1(s_i,s_j),\;\beta_i:=w^2(s_i,a),\;\gamma_i:=w^3(a,s_i),\;\delta:=w^4(a,a),
\end{align}
then the following equations hold
\begin{itemize}
  \item[(i)] $\alpha_{ij}^{k_i}=\alpha_{ij}^{k_j}=1,\;1\leq i,j \leq n$;
  \item[(ii)] $\beta_i^{k_i}=1,\;\beta_i^2=\alpha_{i1}^{m_1}...\alpha_{in}^{m_n},\; 1\leq i\leq n$;
   \item[(iii)]$\gamma_i^{k_i}=1,\;\gamma_i^2=\alpha_{1i}^{m_1}...\alpha_{ni}^{m_n},\; 1\leq i\leq n$;
   \item[(iv)] $\delta^2=\beta_1^{m_1+p_1}...\beta_n^{m_n+p_n}=\gamma_1^{m_1+p_1}...\gamma_n^{m_n+p_n}$;
   \item[(v)] $\alpha_{1i}^{p_1}...\alpha_{ni}^{p_n}=\alpha_{i1}^{p_1}...\alpha_{in}^{p_n}=1,\;\beta_1^{p_1}...\beta_n^{p_n}=\gamma_1^{p_1}...\gamma_n^{p_n}$;
\end{itemize}
\end{proposition}

\begin{proof}
Using (i) of Lemma \ref{lem4.1.a} and (i) of Lemma \ref{lem4.1.b}, we know that $w^1$ is a bicharacter on $S$. Then we have $w^1(s_i,s_j)^{k_i}=w^1(s_i,s_j)^{k_j}=1$ and so (i) holds. To show (ii), we note that $l(X_{s_i})^{k_i}=l(X_1)$ and $r(E_a)^2=r(E_{s_1^{m_1}...s_n^{m_n}})$ and if we use (iii) of Lemma \ref{lem4.1.a}, we get $w^2(s_i,a)^{k_i}=1$ through letting $t=a$ and so we have $\beta_i^{k_i}=1$. Similarly, since
\begin{align*}
r(E_a)^2=\sum\limits_{s\in S}w^2(s,a)^2 e_{s},\;r(E_{s_1^{m_1}...s_n^{m_n}})=\sum\limits_{s\in S}w^1(s,s_1^{m_1}...s_n^{m_n})e_{s},
\end{align*}
so we have $w^2(s_i,a)^2 =w^1(s_i,s_1^{m_1}...s_n^{m_n})$ through letting $s=s_i$. Because we have shown $w^1$ is a bicharacter on $S$, we obtain $w^1(s_i,s_1^{m_1}...s_n^{m_n})=\alpha_{i1}^{m_1}...\alpha_{in}^{m_n}$ and hence (ii) holds. If we consider $R_\varphi$, then we know $R_\varphi$ is also a general solution for $\Bbbk^G\#_{\sigma,\tau}\Bbbk \mathbb{Z}_{2}$ and so $R_\varphi$ such that (ii). Due to the $w'^i( 1\leq i \leq 4)$ of $R_\varphi$ such that $w'^2(s_i,a)=w^3(ab,s_i)$ by definition of $R_\varphi$ and we have $w^3(ab,s_i)=w^3(a,s_i)w^1(b,s_i)$ by (ii) of Lemma \ref{lem4.1.a}, we obtain $w'^2(s_i,a)=w^3(a,s_i)w^1(b,s_i)$. But $w^1(b,s_i)=1$ because of (iii) of Corollary \ref{coro5.1.2}, so $w^3(ab,s_i)=w^3(a,s_i)=\gamma_i$ and hence (iii) holds.  To show (iv), we first show (v). Since (iii) of Corollary \ref{coro5.1.2}, we know $w^1(s_i,b)=w^1(b,s_i)=1$. But we have shown $w^1$ is a bicharacter on $S$ and because $b=s_1^{p_1}...s_n^{p_n}$ by the assumption, we know $\alpha_{1i}^{p_1}...\alpha_{ni}^{p_n}=\alpha_{i1}^{p_1}...\alpha_{in}^{p_n}=1$.
Because of (iv) in Corollary \ref{coro5.1.2}, we have $w^4(ab,ab)=w^4(a,a)$. Using (iv) of Lemma \ref{lem4.1.a}, we get $w^4(ab,ab)=w^2(b,ab)w^4(a,ab)$. With the help of the (iv) of Lemma \ref{lem4.1.b}, we obtain $w^4(a,ab)=w^4(a,a)w^3(a,b)$ and so we have $w^2(b,ab)w^3(a,b)=1$. Due to $w^2(b,ab)=w^2(b,a)w^1(b,b)=w^2(b,a)$ by (ii) of Lemma \ref{lem4.1.a}, we know $w^2(b,a)w^3(a,b)=1$. Since $w^2(b,a)=\beta_1^{p_1}...\beta_n^{p_n}$ and $w^3(a,b)=\gamma_1^{p_1}...\gamma_n^{p_n}$, we get $\beta_1^{p_1}...\beta_n^{p_n}=\gamma_1^{-p_1}...\gamma_n^{-p_n}$. But $w^3(a,b)^2=w^3(a,1)=1$ by (iii) of Lemma \ref{lem4.1.b}, we get $\gamma_1^{2p_1}...\gamma_n^{2p_n}=1$ and thus $\beta_1^{p_1}...\beta_n^{p_n}=\gamma_1^{p_1}...\gamma_n^{p_n}$. Therefore (v) holds. To show (iv), we only need to show $\delta^2=\beta_1^{m_1+p_1}...\beta_n^{m_n+p_n}$ due to the same reason with the proof of (iii). Since $l(X_a)^2=l(X_{a^2})=l(X_{s_1^{m_1}...s_n^{m_n}})$ and the following equations hold
\begin{align*}
l(X_a)^2=\sum\limits_{s\in S}w^4(a,t) w^4(a,t\triangleleft x) e_{t},\;l(X_{s_1^{m_1}...s_n^{m_n}})=\sum\limits_{t\in T}w^2(s_1^{m_1}...s_n^{m_n},t)e_{t},
\end{align*}
we have
\begin{align}
\label{e5.1} w^4(a,a)w^4(a,ab)=w^2(s_1^{m_1}...s_n^{m_n},a)
\end{align}
through letting $t=a$. Since (iii) of Lemma \ref{lem4.1.a}, we get $w^2(s,t)w^2(s',t)=w^2(ss',t)$ for $s\in S,t\in T$. So we have
\begin{align}
\label{e5.2} w^2(s_1^{m_1}...s_n^{m_n},a)=\beta_1^{m_1}...\beta_n^{m_n}.
\end{align}
Since the (iii) in Lemma \ref{lem4.1.2}, we have $w^3(ab,b)=\frac{w^4(a,ab)}{w^4(a,a)}$. Owing to (ii) of Lemma \ref{lem4.1.a}, we can get $w^3(ab,b)=w^3(a,b)w^1(b,b)$. But we have known $w^1(b,b)=1$, so $w^3(ab,b)=w^3(a,b)$. Furthermore, owing to $w^4(a,ab)=w^4(a,a)w^3(a,b)$ and $w^3(a,b)=\gamma_1^{-p_1}...\gamma_n^{-p_n}$, we have
\begin{align}
\label{e5.3} w^4(a,ab)=w^4(a,a)\gamma_1^{-p_1}...\gamma_n^{-p_n}.
\end{align}
Combining equations \eqref{e5.1}-\eqref{e5.3} and (v), we obtain (iv).
\end{proof}
In fact, the four tuple $(\alpha_{ij},\beta_i,\gamma_i,\delta)_{1\leq i,j \leq n}$ in the above proposition completely determine $R$.
\begin{proposition}\label{pro5.1.4}
Let $R$ be stated in Proposition \ref{pro5.1.3} and let $(\alpha_{ij},\beta_i,\gamma_i,\delta)_{1\leq i,j \leq n}$ be defined by \eqref{e5.x} in Proposition \ref{pro5.1.3}, then the following equations hold
\begin{itemize}
  \item[(i)] $w^1(s_1^{i_1}...s_n^{i_n},s_1^{j_1}...s_n^{j_n})=\prod\limits_{k= 1}^{n} \prod\limits_{l = 1}^{n}\alpha_{k l}^{i_k j_l}$;
  \item[(ii)] $w^2(s_1^{i_1}...s_n^{i_n},s_1^{j_1}...s_n^{j_n}a)=(\prod\limits_{k= 1}^{n}\beta_k^{i_k}) \prod\limits_{k= 1}^{n} \prod\limits_{l = 1}^{n}\alpha_{k l}^{i_k j_l}$;
   \item[(iii)]$w^3(s_1^{i_1}...s_n^{i_n}a,s_1^{j_1}...s_n^{j_n})=(\prod\limits_{k= 1}^{n}\gamma_k^{j_k}) \prod\limits_{k=1}^{n} \prod\limits_{l = 1}^{n}\alpha_{k l}^{i_k j_l}$;
   \item[(iv)] $w^4(s_1^{i_1}...s_n^{i_n}a,s_1^{j_1}...s_n^{j_n}a)=(\prod\limits_{k= 1}^{n}\beta_k^{i_k}) (\prod\limits_{k= 1}^{n}\gamma_k^{j_k}) \prod\limits_{k=1}^{n} \prod\limits_{l = 1}^{n}\alpha_{k l}^{i_k j_l}\delta$;
\end{itemize}
where $0 \leq i_1,...,i_n \leq n-1$ and $0 \leq j_1,...,j_n \leq n-1$.
\end{proposition}

\begin{proof}
Since $w^1$ is a bicharacter on $S$, we get (i). Owing to (ii) of Lemma \ref{lem4.1.b}, we know $w^2(s_1^{i_1}...s_n^{i_n},s_1^{j_1}...s_n^{j_n}a)=w^1(s_1^{i_1}...s_n^{i_n},s_1^{j_1}...s_n^{j_n})w^2(s_1^{i_1}...s_n^{i_n},a)$. Due to (iii) of Lemma \ref{lem4.1.a}, we obtain $w^2(s_1^{i_1}...s_n^{i_n},a)=\prod\limits_{k= 1}^{n}\beta_k^{i_k}$ and so we have (ii). Similarly, we can show (iii). Thanks to (iv) of Lemma \ref{lem4.1.a}, we get
\begin{align}
\label{e5.1.4} w^4(s_1^{i_1}...s_n^{i_n}a,s_1^{j_1}...s_n^{j_n}a)=w^2(s_1^{i_1}...s_n^{i_n},s_1^{j_1}...s_n^{j_n}a)w^4(a,s_1^{j_1}...s_n^{j_n}a).
\end{align}
Using (iv) of Lemma \ref{lem4.1.b}, we have
\begin{align}
\label{e5.1.5} w^4(a,s_1^{j_1}...s_n^{j_n}a)=w^3(a\triangleleft x,s_1^{j_1}...s_n^{j_n})w^4(a,a).
\end{align}
Because of \eqref{e3.13} in Lemma \ref{lem4.1.1}, we get
\begin{align}
\label{e5.1.6} w^3(a\triangleleft x,s_1^{j_1}...s_n^{j_n})=w^3(a,s_1^{j_1}...s_n^{j_n}).
\end{align}
Since the equations \eqref{e5.1.4}-\eqref{e5.1.6} and (ii),(iii), we know (iv).
\end{proof}
Conversely, given a four tuple $(\alpha_{ij},\beta_i,\gamma_i,\delta)_{1\leq i,j \leq n}$ satisfying conditions (i)-(v) of Proposition \ref{pro5.1.3}, then we have

\begin{proposition}\label{pro5.1.5}
Given a four tuple $(\alpha_{ij},\beta_i,\gamma_i,\delta)_{1\leq i,j \leq n}$ satisfying conditions (i)-(v) of Proposition \ref{pro5.1.3} and let $R$ be the form (ii) on $\Bbbk^G\#\Bbbk \mathbb{Z}_{2}$ in Proposition \ref{pro2.1.1}. If $w^i(1\leq i\leq 4)$ of $R$ are given by (i)-(iv) in Proposition \ref{pro5.1.4} by using the four tuple, then $R$ is a general solution for $\Bbbk^G\#_{\sigma,\tau}\Bbbk \mathbb{Z}_{2}$.
\end{proposition}

\begin{proof}
Since Corollary \ref{coro5.1.2}, we only need to show $R$ such that the conditions of Corollary \ref{coro5.1.2}. Because the definition of $w^1$, we know $w^1$ is a bicharacter on $S$ and hence we get $l(E_{s_1})^{i_1}...l(E_{s_n})^{i_n}=l(E_{s_1^{i_1}...s_n^{i_n}})$. To show $l(X_{s_1})^{i_1}...l(X_{s_n})^{i_n}=P_{s_1^{i_1}...s_n^{i_n}}l(X_{s_1^{i_1}...s_n^{i_n}})$, we need only to prove that $w^2(-,t)$ is a character on $S$ for $t\in T$ by (iii) of Lemma \ref{lem4.1.a}. By definition of $w^2$, we get $w^2(-,s_1^{j_1}...s_n^{j_n}a)$ is a character on $S$. Owing to $aS=T$, we obtain $w^2(-,t)$ is a character on $S$ for $t\in T$ and so (i) of Lemma \ref{lem5.1.2} holds. To show (ii) of Lemma \ref{lem5.1.2}, we only need to prove that $w^1(s,s')w^3(a,s)=w^3(as,s')$ and $w^2(s,t)w^4(a,t)=w^4(as,t)$ for $s,s'\in S,\;t\in T$ because of (iii), (iv) of Lemma \ref{lem4.1.a}. And these equalities are not difficult to check and so (ii) of Lemma \ref{lem5.1.2} hold. To show (iii) of Lemma \ref{lem5.1.2}, note that $w^1$ is a bicharacter on $S$ and $\alpha_{ij}^{k_i}=1$ by assumption and hence $l(E_{s_i})^{k_i}=l(E_1)$. Similarly, because $w^2(-,t)$ is a character on $S$ for $t\in T$ and $\beta_i^{k_i}=1$ by our conditions, therefore we get $l(X_{s_i})^{k_i}=l(X_1)$ and so we know (iii) of Lemma \ref{lem5.1.2} hold. To show (iv), it can be seen that $w^3(t,-)$ is a bicharacter on $S$ and $\gamma_i^2=\alpha_{1i}^{m_1}...\alpha_{ni}^{m_n}$ by assumption and so we have $l(E_a)^2=l(E_{s_1^{m_1}...s_1^{m_n}})$. By definition, we have
\begin{align*}
l(X_a)^2=[\sum\limits_{t\in T}w^4(a,t)e_{t}x][\sum\limits_{t\in T}w^4(a,t)e_{t}x]=\sum\limits_{t\in T}w^4(a,t)w^4(a,t\triangleleft x)e_{t}
\end{align*}
and $l(X_{s_1^{m_1}...s_1^{m_n}})=\sum_{t\in T}w^2(s_1^{m_1}...s_1^{m_n},t)e_{t}x$. To show $l(X_a)^2=l(X_{s_1^{m_1}...s_1^{m_n}})$, we need only to show $w^4(a,t)w^4(a,t\triangleleft x)=w^2(s_1^{m_1}...s_1^{m_n},t)$ for $t\in T$. For the simplest case $t=a$, we have $w^4(a,a)w^4(a,ab)=\delta^2\prod_{k= 1}^{n}\gamma_k^{p_k}$ and $w^2(s_1^{m_1}...s_1^{m_n},a)=\prod_{k= 1}^{n}\beta_k^{m_k}$ by definition. Since (v) of Proposition \ref{pro5.1.3}, we have $\beta_1^{p_1}...\beta_n^{p_n}=\gamma_1^{p_1}...\gamma_n^{p_n}$. And because $b^2=1$ and $b=s_1^{p_1}...s_n^{p_n}$, we get $\gamma_1^{2p_1}...\gamma_n^{2p_n}=1$. Therefore we get $w^4(a,t)w^4(a,a\triangleleft x)=w^2(s_1^{m_1}...s_1^{m_n},a)$. For the case $t=s_1^{j_1}...s_1^{j_n}a$, if we use the equalities $w^4(a,a)w^4(a,a\triangleleft x)=w^2(s_1^{m_1}...s_1^{m_n},a),\;\gamma_i^2=\alpha_{1_i}^{m_1}...\alpha_{n_i}^{m_n}$, then we can also prove that $w^4(a,s_1^{j_1}...s_1^{j_n}a)w^4(a,s_1^{j_1}...s_1^{j_n}ab)=w^2(s_1^{m_1}...s_1^{m_n},s_1^{j_1}...s_1^{j_n}a)$ and hence we have show (iv) of Lemma \ref{lem5.1.2}. Therefore we have prove that (i) of Corollary \ref{coro5.1.2} holds. To prove that $R$ such that (ii) of Corollary \ref{coro5.1.2}, we consider the $R_\varphi$. If we denote the four maps associated with $R_\varphi$ as $w'^i(1\leq i \leq 4)$, then it can be seen that the following equations hold
\begin{align*}
w'^1(s_i,s_j)=\alpha_{ji},\;w'^2(s_i,a)=\gamma_i,\;w'^3(a,s_i)=\beta_i,\;w'^4(a,a)=\delta.
\end{align*}
Furthermore, one can get that the four tuple $(\alpha_{ji},\gamma_i,\beta_i,\delta)_{1\leq i,j \leq n}$ satisfies conditions (i)-(v) of Proposition \ref{pro5.1.3} and hence $R_\varphi$ such that the conditions of this Proposition and hence $R_\varphi$ such that (i) of Corollary \ref{coro5.1.2}. And this implies that $R$ satisfies (ii) of Corollary \ref{coro5.1.2}.  The (iii) of Corollary \ref{coro5.1.2} holds by (v) of Proposition \ref{pro5.1.3}. Finally we prove that $w^4(ab,ab)=w^4(a,a)$. Due to $\beta_1^{p_1}...\beta_n^{p_n}=\gamma_1^{p_1}...\gamma_n^{p_n},\;\gamma_1^{2p_1}...\gamma_n^{2p_n}=1$ and the definition of $w^4$, we know that $w^4(ab,ab)=w^4(a,a)$.
\end{proof}

Combining Propositions \ref{pro5.1.3}-\ref{pro5.1.5}, we obtain the following theorem
\begin{theorem}\label{thm5.1.1}
All the general solutions for $\Bbbk^G\#_{\sigma,\tau}\Bbbk \mathbb{Z}_{2}$ are given by (i)-(iv) of Proposition \ref{pro5.1.4}, where $(\alpha_{ij},\beta_i,\gamma_i,\delta)_{1\leq i,j \leq n}$ satisfies conditions (i)-(v) of Proposition \ref{pro5.1.3}.
\end{theorem}

\begin{proof}
Since Propositions \ref{pro5.1.3}-\ref{pro5.1.5}, we get what we want.
\end{proof}

Let $(\alpha_{ij},\beta_i,\gamma_i,\delta)_{1\leq i,j \leq n}:=(1,1,1,1)_{1\leq i,j \leq n}$, then we can see that the four tuple such that the conditions of Proposition \ref{pro5.1.3} and so we get a general solution for $\Bbbk^G\#_{\sigma,\tau}\Bbbk \mathbb{Z}_{2}$. And hence we know that the general solution for $\Bbbk^G\#_{\sigma,\tau}\Bbbk \mathbb{Z}_{2}$ always exists.

\subsection{Special solutions for quasitriangular structures on $\Bbbk^G\#_{\sigma,\tau}\Bbbk \mathbb{Z}_{2}$}
In this subsection, we will imitate the method used in the subsection 5.1 to give a necessary and sufficient condition for the existence of a special solution on $\Bbbk^G\#_{\sigma,\tau}\Bbbk \mathbb{Z}_{2}$.

\begin{proposition}\label{pro5.2.1}
Let $R$ be a special solution for $\Bbbk^G\#_{\sigma,\tau}\Bbbk \mathbb{Z}_{2}$, and if we denote
\begin{align*}
\alpha_{ij}:=w^1(s_i,s_j),\;\beta_i:=w^2(s_i,a),\;\gamma_i:=w^3(a,s_i),\;\delta:=w^4(a,a),
\end{align*}
then the following equations hold
\begin{itemize}
  \item[(i)] $\alpha_{ij}^{k_i}=\alpha_{ij}^{k_j}=1,\;1\leq i,j \leq n$;
  \item[(ii)] $\beta_i^{k_i}=P_{s_i^{k_i}},\;\beta_i^2\sigma(s_i)=\alpha_{i1}^{m_1}...\alpha_{in}^{m_n},\; 1\leq i\leq n$;
   \item[(iii)]$\gamma_i^{k_i}=P_{s_i^{k_i}},\;\gamma_i^2\sigma(s_i)=\alpha_{1i}^{m_1}...\alpha_{ni }^{m_n},\; 1\leq i\leq n$;
   \item[(iv)] $\delta^2=[\tau(a,a)\tau(b,a)\tau(b,b)^{-1}\sigma(a)^{-1} P_{s_1^{m_1}...s_n^{m_n}}^{-1}P_{s_1^{p_1}...s_n^{p_n}}^{-1}]\beta_1^{m_1+p_1}...\beta_n^{m_n+p_n}$;
   \item[(v)]  $\delta^2=[\tau(a,a)\tau(a,b)\tau(b,b)^{-1}\sigma(a)^{-1} P_{s_1^{m_1}...s_n^{m_n}}^{-1}P_{s_1^{p_1}...s_n^{p_n}}^{-1}]\gamma_1^{m_1+p_1}...\gamma_n^{m_n+p_n}$;
   \item[(vi)]$\alpha_{1i}^{p_1}...\alpha_{ni}^{p_n}=\alpha_{i1}^{p_1}...\alpha_{in}^{p_n}=\eta(a,s_i),\;\beta_1^{p_1}...\beta_n^{p_n}=\gamma_1^{p_1}...\gamma_n^{p_n}\eta(a,b)$;
\end{itemize}
\end{proposition}

\begin{proof}
We mimic the proof of Proposition \ref{pro5.1.3} as follows. Since $w^1$ is a bicharacter on $S$, we have $w^1(s_i,s_j)^{k_i}=w^1(s_i,s_j)^{k_j}=1$ and so (i) holds. To show (ii), we note that $l(X_{s_i})^{k_i}=P_{s_i^{k_i}}l(X_1)$ and $r(E_a)^2=r(E_{s_1^{m_1}...s_n^{m_n}})$, if we use (iii) of Lemma \ref{lem4.1.a}, we get $w^2(s_i,a)^{k_i}=P_{s_i^{k_i}}$ and so we have $\beta_i^{k_i}=P_{s_i^{k_i}}$. Similarly, since
\begin{align*}
r(E_a)^2=\sum\limits_{s\in S}w^2(s,a)^2 \sigma(s) e_{s},\;r(E_{s_1^{m_1}...s_n^{m_n}})=\sum\limits_{s\in S}w^1(s,s_1^{m_1}...s_n^{m_n})e_{s},
\end{align*}
we have $w^2(s_i,a)^2 \sigma(s_i)=w^1(s_i,s_1^{m_1}...s_n^{m_n})$ through letting $s=s_i$. Because we have shown $w^1$ is a bicharacter on $S$, we obtain $w^1(s_i,s_1^{m_1}...s_n^{m_n})=\alpha_{i1}^{m_1}...\alpha_{in}^{m_n}$ and hence (ii) holds. If we consider $R_\varphi$, then we know $R_\varphi$ is also a special solution for $\Bbbk^G\#_{\sigma,\tau}\Bbbk \mathbb{Z}_{2}$ and so $R_\varphi$ such that (ii). Due to the $w'^i( 1\leq i \leq 4)$ of $R_\varphi$ such that $w'^2(s_i,a)=w^3(ab,s_i)$ by definition of $R_\varphi$ and $w^3(ab,s_i)=w^3(a,s_i)w^1(b,s_i)$ by (ii) of Lemma \ref{lem4.1.a}, we obtain $w'^2(s_i,a)=w^3(a,s_i)w^1(b,s_i)$. But $w^1(b,s_i)=\eta(a,s_i)$ because of (iii) in Corollary \ref{coro5.1.2}, so $w^3(ab,s_i)=w^3(a,s_i)\eta(a,s_i)=\gamma_i\eta(a,s_i)$. Because $\eta(a,s_i)^{k_i}=\eta(a,s_i)^2=1$ and (ii) holds for $R_\varphi$, we know (iii) holds. To show (iv), we first show (vi). Since (iii) in Corollary \ref{coro5.1.2}, we know $w^1(s_i,b)=w^1(b,s_i)=\eta(a,s_i)$. But we have known $w^1$ is a bicharacter on $S$ and $b=s_1^{p_1}...s_n^{p_n}$ by the assumption, we know $\alpha_{1i}^{p_1}...\alpha_{ni}^{p_n}=\alpha_{i1}^{p_1}...\alpha_{in}^{p_n}=\eta(a,s_i)$.
Because of (iv) in Corollary \ref{coro5.1.2}, we have $w^4(ab,ab)=\frac{\tau(ab,ab)}{\tau(a,a)}w^4(a,a)$. Using (iv) of Lemma \ref{lem4.1.a}, we get $w^4(ab,ab)=\tau(b,a)^{-1}w^2(b,ab)w^4(a,ab)$. With the help of the (iv) of Lemma \ref{lem4.1.b}, we obtain $w^4(a,ab)=\tau(a,b)^{-1}w^4(a,a)w^3(a,b)$ and so we have $\tau(b,a)^{-1}\tau(a,b)^{-1}w^2(b,ab)w^3(a,b)=\frac{\tau(ab,ab)}{\tau(a,a)}$. Due to $w^2(b,ab)=\eta(b,a)w^2(b,a)$ by \eqref{e3.12} of Lemma \ref{lem4.1.1}, we know $\tau(b,a)^{-1}\tau(a,b)^{-1}\eta(b,a)w^2(b,a)w^3(a,b)=\frac{\tau(ab,ab)}{\tau(a,a)}$. It can be seen that $\frac{\tau(b,a)\tau(a,b)\tau(ab,ab)}{\tau(a,a)\eta(b,a)}=\tau(b,b)\eta(a,b)$. Moreover, since $w^2(b,a)=P_{s_1^{p_1}...s_n^{p_n}}^{-1}\beta_1^{p_1}...\beta_n^{p_n}$ and $w^3(a,b)=P_{s_1^{p_1}...s_n^{p_n}}^{-1}\gamma_1^{p_1}...\gamma_n^{p_n}$, we get $(\beta_1^{p_1}...\beta_n^{p_n})(\gamma_1^{p_1}...\gamma_n^{p_n})=\tau(b,b)\eta(a,b)P_{s_1^{p_1}...s_n^{p_n}}^2$.
Due to (iii) of Lemma \ref{lem4.1.a}, we obtain $w^3(a,b)^2=\tau(b,b)$ and hence $w^3(a,b)=\tau(b,b)w^3(a,b)^{-1}=\tau(b,b)P_{s_1^{p_1}...s_n^{p_n}}\gamma_1^{-p_1}...\gamma_n^{-p_n}$. Therefore (vi) holds. To show (iv), since $l(X_a)^2=\tau(a,a)l(X_{a^2})=\tau(a,a)l(X_{s_1^{m_1}...s_n^{m_n}})$ and the following equations hold
\begin{align*}
l(X_a)^2=\sum\limits_{s\in S}w^4(a,t) w^4(a,t\triangleleft x)\sigma(t) e_{t},\;l(X_{s_1^{m_1}...s_n^{m_n}})=\sum\limits_{t\in T}w^2(s_1^{m_1}...s_n^{m_n},t)e_{t},
\end{align*}
we have
\begin{align}
\label{e5.2.1} w^4(a,a)w^4(a,ab)\sigma(a)=\tau(a,a)w^2(s_1^{m_1}...s_n^{m_n},a)
\end{align}
through letting $t=a$. Since Lemma \ref{lem4.1.a}, we get $w^2(s,t)w^2(s',t)=\tau(s,s')w^2(ss',t)$ for $s\in S,t\in T$, so we have
\begin{align}
\label{e5.2.2} w^2(s_1^{m_1}...s_n^{m_n},a)=P_{s_1^{m_1}...s_n^{m_n}}^{-1}\beta_1^{m_1}...\beta_n^{m_n}.
\end{align}
Since the (iii) in Lemma \ref{lem4.1.2}, we have $w^3(ab,b)=\tau(b,a)\frac{w^4(a,ab)}{w^4(a,a)}$. Owing to \eqref{e3.13} in Lemma \ref{lem4.1.1}, we can get $w^3(ab,b)=w^3(a,b)\eta(a,b)$. Because we have known $w^3(a,b)=\tau(b,b)P_{s_1^{p_1}...s_n^{p_n}}\gamma_1^{-p_1}...\gamma_n^{-p_n}$
and $w^3(ab,b)=\tau(b,a)\frac{w^4(a,ab)}{w^4(a,a)}$, we have
\begin{align}
\label{e5.2.3} w^4(a,ab)=w^4(a,a)\tau(b,a)^{-1}\eta(a,b)\tau(b,b)P_{s_1^{p_1}...s_n^{p_n}}\gamma_1^{-p_1}...\gamma_n^{-p_n}.
\end{align}
Combining equations \eqref{e5.2.1}-\eqref{e5.2.3} and (v), we obtain (iv). Finally, if we consider $R_\varphi$, then we have (iv) holds for $R_\varphi$. Therefore we get (v) holds for $R$.
\end{proof}

Similar to Proposition \ref{pro5.1.4}, we have
\begin{proposition}\label{pro5.2.4}
Let $R$ be stated in Proposition \ref{pro5.2.1} and let $(\alpha_{ij},\beta_i,\gamma_i,\delta)_{1\leq i,j \leq n}$ be defined in Proposition \ref{pro5.2.1}, then the following equations hold
\begin{itemize}
  \item[(i)] $w^1(s_1^{i_1}...s_n^{i_n},s_1^{j_1}...s_n^{j_n})=\prod\limits_{k= 1}^{n} \prod\limits_{l = 1}^{n}\alpha_{k l}^{i_k j_l}$;
  \item[(ii)] $w^2(s_1^{i_1}...s_n^{i_n},s_1^{j_1}...s_n^{j_n}a)=P_{s_1^{i_1}...s_n^{i_n}}^{-1}(\prod\limits_{k= 1}^{n}\beta_k^{i_k}) \prod\limits_{k= 1}^{n} \prod\limits_{l = 1}^{n}\alpha_{k l}^{i_k j_l}$;
   \item[(iii)]$w^3(s_1^{i_1}...s_n^{i_n}a,s_1^{j_1}...s_n^{j_n})=P_{s_1^{j_1}...s_n^{j_n}}^{-1}(\prod\limits_{k= 1}^{n}\gamma_k^{j_k}) \prod\limits_{k=1}^{n} \prod\limits_{l = 1}^{n}\alpha_{k l}^{i_k j_l}$;
   \item[(iv)] $w^4(s_1^{i_1}...s_n^{i_n}a,s_1^{j_1}...s_n^{j_n}a)=\lambda(i_1,...,i_n,j_1,...,j_n)(\prod\limits_{k= 1}^{n}\beta_k^{i_k}) (\prod\limits_{k= 1}^{n}\gamma_k^{j_k}) \prod\limits_{k=1}^{n} \prod\limits_{l = 1}^{n}\alpha_{k l}^{i_k j_l}\delta$;
\end{itemize}
where $0 \leq i_1,...,i_n \leq n-1,\;0 \leq j_1,...,j_n \leq n-1$ and $\lambda(i_1,...,i_n,j_1,...,j_n):=P_{s_1^{i_1}...s_n^{i_n}}^{-1}P_{s_1^{j_1}...s_n^{j_n}}^{-1}\tau(s_1^{i_1}...s_n^{i_n},a)^{-1}\tau(a,s_1^{j_1}...s_n^{j_n})^{-1}$.
\end{proposition}

\begin{proof}
Since $w^1$ is a bicharacter on $S$, we get (i). Owing to (ii) of Lemma \ref{lem4.1.b}, we know $w^2(s_1^{i_1}...s_n^{i_n},s_1^{j_1}...s_n^{j_n}a)=w^1(s_1^{i_1}...s_n^{i_n},s_1^{j_1}...s_n^{j_n})w^2(s_1^{i_1}...s_n^{i_n},a)$. Due to (iii) of Lemma \ref{lem4.1.a}, we obtain $w^2(s_1^{i_1}...s_n^{i_n},a)=P_{s_1^{p_1}...s_n^{p_n}}^{-1}\prod\limits_{k= 1}^{n}\beta_k^{i_k}$ and so we have (ii). Similarly, we can show (iii). Thanks to (iv) of Lemma \ref{lem4.1.a}, we get
\begin{align}
\label{e5.4} w^4(s_1^{i_1}...s_n^{i_n}a,s_1^{j_1}...s_n^{j_n}a)=\tau(s_1^{i_1}...s_n^{i_n},a)^{-1}w^2(s_1^{i_1}...s_n^{i_n},s_1^{j_1}...s_n^{j_n}a)w^4(a,s_1^{j_1}...s_n^{j_n}a).
\end{align}
Using (iv) of Lemma \ref{lem4.1.b}, we have
\begin{align}
\label{e5.5} w^4(a,s_1^{j_1}...s_n^{j_n}a)=\tau(s_1^{j_1}...s_n^{j_n},a)^{-1}w^3(a\triangleleft x,s_1^{j_1}...s_n^{j_n})w^4(a,a).
\end{align}
Because of \eqref{e3.13} in Lemma \ref{lem4.1.1}, we get
\begin{align}
\label{e5.6} w^3(a\triangleleft x,s_1^{j_1}...s_n^{j_n})=\eta(a,s_1^{j_1}...s_n^{j_n})w^3(a,s_1^{j_1}...s_n^{j_n}).
\end{align}
Since the equations \eqref{e5.4}-\eqref{e5.6} and (ii),(iii), we know (iv) holds.
\end{proof}

Conversely, given a four tuple $(\alpha_{ij},\beta_i,\gamma_i,\delta)_{1\leq i,j \leq n}$ satisfying conditions (i)-(vi) of Proposition \ref{pro5.2.1}, then we have
\begin{proposition}\label{pro5.2.5}
Let $R$ be the form (ii) on $\Bbbk^G\#_{\sigma,\tau}\Bbbk \mathbb{Z}_{2}$ in Proposition \ref{pro2.1.1}, and if $w^i(1\leq i\leq 4)$ of $R$ are given by (i)-(iv) in Proposition \ref{pro5.2.4} by using the four tuple above, then $R$ is a special solution for $\Bbbk^G\#_{\sigma,\tau}\Bbbk \mathbb{Z}_{2}$.
\end{proposition}

\begin{proof}
Since Corollary \ref{coro5.1.2}, we only need to show $R$ such that the conditions of Corollary \ref{coro5.1.2}. Because the definition of $w^1$, we know $w^1$ is a bicharacter on $S$ and hence we get $l(E_{s_1})^{i_1}...l(E_{s_n})^{i_n}=l(E_{s_1^{i_1}...s_n^{i_n}})$. To show $l(X_{s_1})^{i_1}...l(X_{s_n})^{i_n}=P_{s_1^{i_1}...s_n^{i_n}}l(X_{s_1^{i_1}...s_n^{i_n}})$, we only need to show $\prod_{k = 1}^{n}w^2(s_k,t)^{i_k}=P_{s_1^{i_1}...s_n^{i_n}}w^2(s_1^{i_1}...s_n^{i_n},t)$ for $t\in T$. Owing to $aS=T$, we can assume $t=s_1^{j_1}...s_n^{j_n}a$. Since $w^2(s_k,s_1^{j_1}...s_n^{j_n}a)=\beta_k \prod_{l = 1}^{n} \alpha_{kl}^{j_l}$, we obtain
\begin{align}
\label{w5.1} \prod_{k = 1}^{n}w^2(s_k,s_1^{j_1}...s_n^{j_n}a)^{i_k}=\prod_{k = 1}^{n}\beta_k^{i_k}\prod_{k,l = 1}^{n}\alpha_{kl}^{i_kj_l}.
\end{align}
Because the definition of $w^2$ and the equation \eqref{w5.1} above, we get $\prod_{k = 1}^{n}w^2(s_k,t)^{i_k}=P_{s_1^{i_1}...s_n^{i_n}}w^2(s_1^{i_1}...s_n^{i_n},t)$ for $t\in T$ and so (i) of Lemma \ref{lem5.1.2} holds. To show (ii) of Lemma \ref{lem5.1.2}, we only need to prove that $w^1(s,s')w^3(a,s)=w^3(as,s')$ and $w^2(s,t)w^4(a,t)=\tau(s,a)w^4(as,t)$ for $s,s'\in S,\;t\in T$ due to Lemma \ref{lem4.1.a}. And these equalities are not difficult to check and so (ii) of Lemma \ref{lem5.1.2} hold. To show (iii) of Lemma \ref{lem5.1.2}, note that $w^1$ is a bicharacter on $S$ and $\alpha_{ij}^{k_i}=1$ by assumption and hence $l(E_{s_i})^{k_i}=l(E_1)$. To show $l(X_{s_i})^{k_i}=P_{s_i^{k_i}}l(X_1)$, we only need to prove that $w^2(s_i,t)^{k_i}=P_{s_i^{k_i}}$ for $t\in T$. Since the definition of $w^2$ and $\beta_i^{k_i}=P_{s_i^{k_i}}$ by assumption, we get $w^2(s_i,t)^{k_i}=P_{s_i^{k_i}}$ for $t\in T$ and so we know (iii) of Lemma \ref{lem5.1.2} hold. To show (iv), it can be seen that $l(E_a)^2=l(E_{s_1^{m_1}...s_1^{m_n}})$ is equivalent to $w^3(a,s)^2\sigma(s)=w^1(a^2,s)$ for $s\in S$. Since (iii) in Proposition \ref{pro5.2.1} and $w^1$ is a bicharacter, we know $w^3(a,s_i)^2\sigma(s_i)=w^1(a^2,s_i)$. By definition, we have $w^3(a,s_1^{j_1}...s_n^{j_n})=P_{s_1^{j_1}...s_n^{j_n}}^{-1}\prod_{k=1}^n w^3(a,s_k)^{j_k}$
and so we get
\begin{align*} w^3(a,s_1^{j_1}...s_n^{j_n})^2&=P_{s_1^{j_1}...s_n^{j_n}}^{-2}\prod_{k=1}^n w^3(a,s_k)^{2j_k}\\
&=P_{s_1^{j_1}...s_n^{j_n}}^{-2}\prod_{k=1}^n w^1(a^2,s_k)^{j_k}\sigma(s_k)^{-j_k}\\
&=P_{s_1^{j_1}...s_n^{j_n}}^{-2}(\prod_{k=1}^n\sigma(s_k)^{-j_k})w^1(a^2,s_1^{j_1}...s_n^{j_n})
\end{align*}
To show $w^3(a,s_1^{j_1}...s_n^{j_n})^2\sigma(s_1^{j_1}...s_n^{j_n})=w^1(a^2,s_1^{j_1}...s_n^{j_n})$, we only need to show
$$P_{s_1^{j_1}...s_n^{j_n}}^{-2}\prod_{k=1}^n \sigma(s_k)^{-j_k}\sigma(s_1^{j_1}...s_n^{j_n})=1.$$
But the equation above follows from the following Lemma \ref{lem5.2.1} and so we have $l(E_a)^2=l(E_{s_1^{m_1}...s_1^{m_n}})$. By definition, we have
\begin{align*}
l(X_a)^2=[\sum\limits_{t\in T}w^4(a,t)e_{t}x][\sum\limits_{t\in T}w^4(a,t)e_{t}x]=\sum\limits_{t\in T}w^4(a,t)w^4(a,t\triangleleft x)\sigma(t)e_{t}
\end{align*}
and $l(X_{s_1^{m_1}...s_1^{m_n}})=\sum_{t\in T}w^2(s_1^{m_1}...s_1^{m_n},t)e_{t}x$. To show $l(X_a)^2=\tau(a,a)l(X_{s_1^{m_1}...s_1^{m_n}})$, we only need to show $w^4(a,t)w^4(a,t\triangleleft x)\sigma(t)=\tau(a,a)w^2(s_1^{m_1}...s_1^{m_n},t)$ for $t\in T$. For the simplest case $t=a$, we have $w^4(a,a)w^4(a,ab)=\tau(a,b)^{-1}P_{s_1^{p_1}...s_n^{p_n}}^{-1}\prod_{k= 1}^{n}\gamma_k^{p_k}\delta^2$ and $w^2(s_1^{m_1}...s_1^{m_n},a)=P_{s_1^{m_1}...s_n^{m_n}}^{-1}\prod_{k= 1}^{n}\beta_k^{m_k}$ by definition. Since the proof of Proposition \ref{pro5.2.1}, we have $w^3(a,b)=P_{s_1^{p_1}...s_n^{p_n}}^{-1}\gamma_1^{p_1}...\gamma_n^{p_n}=\tau(b,b)P_{s_1^{p_1}...s_n^{p_n}}\gamma_1^{-p_1}...\gamma_n^{-p_n}$
and $w^4(a,a)w^4(a,ab)=\tau(a,b)^{-1}\tau(b,b)P_{s_1^{p_1}...s_n^{p_n}}\gamma_1^{-p_1}...\gamma_n^{-p_n}\delta^2$. Owing to (iv), (vi) of Proposition \ref{pro5.2.1}, we get $w^4(a,a)w^4(a,ab)\sigma(a)=\tau(a,a)w^2(s_1^{m_1}...s_1^{m_n},a)$.
For the case $t=s_1^{j_1}...s_1^{j_n}a$, using the following equalities
\begin{align*}
w^4(a,a)w^4(a,ab)\sigma(a)= \tau(a,a)w^2(s_1^{m_1}...s_1^{m_n},a),\;\gamma_i^2\sigma(s_i)=\alpha_{1_i}^{m_1}...\alpha_{n_i}^{m_n}
\end{align*}
and Lemma \ref{lem5.2.1}, we can prove that $w^4(a,t)w^4(a,t\triangleleft x)\sigma(t)=\tau(a,a)w^2(s_1^{m_1}...s_1^{m_n},t)$ and hence we have show (iv) of Lemma \ref{lem5.1.2}. To prove that $R$ such that (i)-(iv) of Lemma \ref{lem5.1.3}, we consider the $R_\varphi$. If we denote the four maps associated with $R_\varphi$ as $w'^i(1\leq i \leq 4)$, then it can be seen that the following equations hold
\begin{align*}
w'^1(s_i,s_j)=\alpha_{ji},\;w'^2(s_i,a)=\eta(a,s_i)\gamma_i,\;w'^3(a,s_i)=\eta(s_i,a)\beta_i,\;w'^4(a,a)=\delta.
\end{align*}
Furthermore, one can check that the four tuple $(\alpha_{ji},\eta(a,s_i)\gamma_i,\eta(s_i,a)\beta_i,\delta)_{1\leq i,j \leq n}$ satisfies conditions (i)-(vi) of Proposition \ref{pro5.2.1} and hence $R_\varphi$ such that the conditions of this Proposition. Therefore $R_\varphi$ such that (i)-(iv) of Lemma \ref{lem5.1.2}. And this implies that $R$ such that (i)-(iv) of Lemma \ref{lem5.1.3}. Since $w^1$ is a bicharacter on $S$ and (vi) of Proposition \ref{pro5.2.1}, we know (iii) of Corollary \ref{coro5.1.2} holds. Finally we prove that $w^4(ab,ab)=\frac{\tau(ab,ab)}{\tau(a,a)}w^4(a,a)$. Since the definition of $w^4$, we have
$$w^4(ab,ab)=P_{s_1^{p_1}...s_n^{p_n}}^{-2}\tau(b,a)^{-1}\tau(a,b)^{-1}\prod_{k=1}^n \beta^{p_k}\prod_{k=1}^n \gamma^{p_k}\eta(a,b)\delta$$.

Due to $P_{s_1^{p_1}...s_n^{p_n}}^{-1}\gamma_1^{p_1}...\gamma_n^{p_n}=\tau(b,b)P_{s_1^{p_1}...s_n^{p_n}}\gamma_1^{-p_1}...\gamma_n^{-p_n}$ and  $\beta_1^{p_1}...\beta_n^{p_n}=\eta(a,b)\gamma_1^{p_1}...\gamma_n^{p_n}$, we know that $w^4(ab,ab)=\frac{\tau(b,b)}{\tau(a,b)\tau(b,a)}w^4(a,a)$. Using the fact $\tau$ is a 2-cocycle, we can see that $\frac{\tau(b,b)}{\tau(a,b)\tau(b,a)}=\frac{\tau(ab,ab)}{\tau(a,a)}$ and hence we get $w^4(ab,ab)=\frac{\tau(ab,ab)}{\tau(a,a)}w^4(a,a)$.
\end{proof}
The following lemma is used to help the proof of Proposition \ref{pro5.2.5}.
\begin{lemma}\label{lem5.2.1}
$P_{s_1^{j_1}...s_n^{j_n}}^{-2}\prod_{k=1}^n \sigma(s_k)^{-j_k}\sigma(s_1^{j_1}...s_n^{j_n})=1$, where $j_1,...,j_n\in \mathbb{N}$.
\end{lemma}

\begin{proof}
On the one hand, it can be seen that $\Delta(E_s)=E_s\otimes E_s+\sigma(s)X_s\otimes X_s$ for $s\in S$ and so we have $\Delta(E_{s_1^{j_1}...s_n^{j_n}})=E_{s_1^{j_1}...s_n^{j_n}}\otimes E_{s_1^{j_1}...s_n^{j_n}}+\sigma(s_1^{j_1}...s_n^{j_n})X_{s_1^{j_1}...s_n^{j_n}}\otimes X_{s_1^{j_1}...s_n^{j_n}}$. On the other hand, we have $\Delta(E_{s_1})^{j_1}...\Delta(E_{s_n})^{j_n}=[E_{s_1}\otimes E_{s_1}+\sigma(s_1)X_{s_1}\otimes X_{s_1}]^{j_1}...[E_{s_n}\otimes E_{s_n}+\sigma(s_n)X_{s_n}\otimes X_{s_n}]^{j_n}=E_{s_1^{j_1}...s_n^{j_n}}\otimes E_{s_1^{j_1}...s_n^{j_n}}+\sigma(s_1)^{j_1}...\sigma(s_n)^{j_n}P_{s_1^{j_1}...s_n^{j_n}}^{2}X_{s_1^{j_1}...s_n^{j_n}}\otimes X_{s_1^{j_1}...s_n^{j_n}}$. Since $E_{s_1}^{j_1}...E_{s_n}^{j_n}=E_{s_1^{j_1}...s_n^{j_n}}$, we get $P_{s_1^{j_1}...s_n^{j_n}}^{2}\prod_{k=1}^n \sigma(s_k)^{j_k}=\sigma(s_1^{j_1}...s_n^{j_n})$.
\end{proof}

By Propositions \ref{pro5.2.1}-\ref{pro5.2.5}, we get the following theorem
\begin{theorem}\label{thm5.2.1}
There exists a quasitriangular structure for $\Bbbk^G\#_{\sigma,\tau}\Bbbk \mathbb{Z}_{2}$ if and only if there exists a four tuple $(\alpha_{ij},\beta_i,\gamma_i,\delta)_{1\leq i,j \leq n}$ satisfies conditions (i)-(vi) of Proposition \ref{pro5.2.1}.
\end{theorem}

\begin{proof}
Since Propositions \ref{pro5.2.1}-\ref{pro5.2.5}, we get what we want.
\end{proof}

\begin{remark}\emph{From the above theorem, we know that all non-trivial quasitriangular structures on $H_G$ are given by (i)-(iv) of Proposition \ref{pro5.2.4}, where $(\alpha_{ij},\beta_i,\gamma_i,\delta)_{1\leq i,j \leq n}$ satisfying conditions (i)-(vi) of Proposition \ref{pro5.2.1}.}\end{remark}

\begin{corollary} \label{coro5.2.1}
There exists a quasitriangular structure for $\Bbbk^G\#_{\sigma,\tau}\Bbbk \mathbb{Z}_{2}$ if and only if there exists a bicharacter $w^1$ on $S$ and a pairing $(\beta_i,\gamma_i)_{1\leq i \leq n}$ satisfies the following conditions
\begin{itemize}
  \item[(i)] $\beta_i^{k_i}=P_{s_i^{k_i}},\;\gamma_i^{k_i}=P_{s_i^{k_i}}$;
  \item[(ii)] $\beta_1^{p_1}...\beta_n^{p_n}=\gamma_1^{p_1}...\gamma_n^{p_n}\eta(a,b),\;\beta_1^{m_1}...\beta_n^{m_n}=\gamma_1^{m_1}...\gamma_n^{m_n}$;
   \item[(iii)]$w^1(s_i,b)=w^1(b,s_i)=\eta(a,s_i),\;w^1(s_i,a^2)=\beta_i^2\sigma(s_i),\;w^1(a^2,s_i)=\gamma_i^2\sigma(s_i)$;
\end{itemize}
where $1\leq i \leq n$.
\end{corollary}

\begin{proof}
If there exists a quasitriangular structure, then we know $w^1, \;(\beta_i,\gamma_i)_{1\leq i \leq n}$ of Proposition \ref{pro5.2.1} such that the conditions (i)-(iii). Conversely, let $\alpha_{ij}:=w^1(s_i,s_j)$ and let $\delta$ be given by (iv) of Proposition \ref{pro5.2.1}, then we know that $(\alpha_{ij},\beta_i,\gamma_i,\delta)_{1\leq i,j \leq n}$ satisfies conditions (i)-(vi) of Proposition \ref{pro5.2.1} by our conditions (i)-(iii). And hence there exists a quasitriangular structure by Theorem \ref{thm5.2.1}.
\end{proof}

\begin{corollary} \label{coro5.2.2}
If there is a bicharacter $w^1$ on $S$ and a set $\{\beta_i\in \Bbbk|\;1\leq i \leq n\}$ such that the following conditions
\begin{itemize}
  \item[(i)] $\beta_i^{k_i}=P_{s_i^{k_i}}$;
  \item[(ii)] $w^1(s_i,a^2)=w^1(a^2,s_i)=\beta_i^2\sigma(s_i)$;
   \item[(iii)]$w^1(s_i,b)=w^1(b,s_i)=\eta(a,s_i)$;
\end{itemize}
where $1\leq i \leq n$, then there exists a quasitriangular structure for $\Bbbk^G\#_{\sigma,\tau}\Bbbk \mathbb{Z}_{2}$.
\end{corollary}

\begin{proof}
Let $\gamma_i:=\beta_i\eta(a,s_i)$, then we can see that $w^1$ and $(\beta_i,\gamma_i)_{1\leq i \leq n}$ satisfy the conditions of Corollary \ref{coro5.2.1}. And so we get what we want.
\end{proof}

\begin{example}\label{def5.2.1}
\emph{ Let $K(8n,\sigma,\tau)$ be in Example \ref{def2.1.1}, then we can assume $s_1=a^2,\;s_2=b$. It can be seen that we can give a bicharacter on $S$ satisfying the conditions of Corollary \ref{coro5.2.2} through the following equations}
\begin{align*}
w^1(s_1,s_1):=\beta^2\sigma(s_1),\;w^1(s_1,s_2)=w^1(s_2,s_1):=\eta(a,s_1),\;w^1(s_2,s_2):=\eta(a,s_2),
\end{align*}
\emph{where $\beta\in \Bbbk$ such that $\beta^n=P_{s_1^{n}}$. That is to say there is a special solution for $K(8n,\sigma,\tau)$.}
\end{example}

\section{$\varphi$-symmetric quasitriangular structures on $\Bbbk^G\#_{\sigma,\tau}\Bbbk \mathbb{Z}_{2}$}
Let $\varphi$ be the Hopf isomorphism in Proposition \ref{pro3.1.3}. By Corollary \ref{coro3.1.2}, we know the most simple quasitriangular structures on $\Bbbk^G\#_{\sigma,\tau}\Bbbk \mathbb{Z}_{2}$ are $\varphi$-symmetric quasitriangular structures. We will give a necessary and sufficient condition for the existence of $\varphi$-symmetric quasitriangular structures on $\Bbbk^G\#_{\sigma,\tau}\Bbbk \mathbb{Z}_{2}$ in this section. Before this, we give the following definition.

\begin{definition}\label{def6.1.1}
A quasitriangular function $w$ on $\Bbbk^G\#_{\sigma,\tau}\Bbbk \mathbb{Z}_{2}$ is called a \emph{$\varphi$-symmetric quasitriangular function} if it satisfies $w(t_1,t_2)=w(t_2,t_1)$ for $t_1,t_2\in T$.
\end{definition}
The following proposition is the reason why we give the above definition.

\begin{proposition}\label{pro6.1.1}
Let $R$ be the form (ii) in Proposition \ref{pro2.1.1}, then $R$ is a $\varphi$-symmetric quasitriangular structure if and only if $w^4$ is a $\varphi$-symmetric quasitriangular function and $w^i(1\leq i \leq 3)$ are given by (i)-(iii) in Lemma \ref{lem4.1.2}.
\end{proposition}

\begin{proof}
If $R$ is a $\varphi$-symmetric quasitriangular structure, then we know $w^4$ is a quasitriangular function due to Proposition \ref{pro4.1.1}. By definition of $\varphi$-symmetric quasitriangular structure, we get $w(t_1,t_2)=w(t_2,t_1)$ for $t_1,t_2\in T$. Moreover, since Lemma \ref{lem4.1.2}, we obtain that $w^i(1\leq i \leq 3)$ are given by (i)-(iii) in Lemma \ref{lem4.1.2} and so we have proved the necessity. Conversely, if $w^4$ is a $\varphi$-symmetric quasitriangular function and $w^i(1\leq i \leq 3)$ are given by (i)-(iii) in Lemma \ref{lem4.1.2}, then we get $R$ is a quasitriangular structure by Theorem \ref{thm4.1.1}. To show $R$ is a $\varphi$-symmetric quasitriangular structure, we only need to prove that $w^2(s,t)=w^3(t\triangleleft x,s)$ and $w^1(s_1,s_2)=w^1(s_2,s_1)$ for $s,s_1,s_2\in S,\;t\in T$ by definition. Since (i)-(iii) of Lemma \ref{lem4.1.2} and $w^4(t_1,t_2)=w^4(t_2,t_1)$ for $t_1,t_2\in T$, we get $w^2(s,t)=w^3(t\triangleleft x,s)$ and $w^1(s_1,s_2)=w^1(s_2,s_1)$.
\end{proof}

\begin{corollary} \label{coro6.1.1}
Let $R$ be a quasitriangular structure on $\Bbbk^G\#_{\sigma,\tau}\Bbbk \mathbb{Z}_{2}$, then $R$ is a $\varphi$-symmetric quasitriangular structure if and only if $w^1(s_i,s_j)=w^1(s_j,s_i),\;w^2(s_i,a)=w^3(ab,s_i)$ for $1\leq i,j\leq n$.
\end{corollary}

\begin{proof}
The necessity is obvious. In order to prove the sufficiency, we only need to prove that $w(t_1,t_2)=w(t_2,t_1)$ for $t_1,t_2\in T$ because of Proposition \ref{pro6.1.1} above. Since $aS=T$, we can assume that $t_1=as$ and $t_2=as'$ for some $s,s'\in S$. Then we have $w^4(as,as')=\tau(s,a)^{-1}w^2(s,as')w^4(a,as')$ by (ii) of Lemma \ref{lem4.1.2}. Because (ii) of Lemma \ref{lem4.1.b}, we have $w^2(s,as')=w^2(s,a)w^1(s,s')$. Owing to $w^4(a,as')=\tau(s',a)^{-1}w^4(a,a)w^3(ab,s')$ by (iii) of Lemma \ref{lem4.1.2}, we get $$w^4(as,as')=\tau(s,a)^{-1}\tau(s',a)^{-1}w^4(a,a)w^1(s,s')w^2(s,a)w^3(ab,s').$$
Since $w^1$ is a bicharacter on $S$, we get $w^1(s,s')=w^1(s',s)$. Due to (iii) of Lemma \ref{lem4.1.a} and (iii) of Lemma \ref{lem4.1.b}, we know $w^2(s,a)=w^3(ab,s)$ and $w^3(ab,s')=w^2(s',a)$. Therefore we have $ w^4(as,as')=w^4(as',as)$.
\end{proof}

As an application of results in subsection 5.2, we give the following proposition

\begin{proposition}\label{pro6.1.2}
There exists a $\varphi$-symmetric quasitriangular structure for $\Bbbk^G\#_{\sigma,\tau}\Bbbk \mathbb{Z}_{2}$ if and only if there exists a bicharacter $w^1$ on $S$ and a set $\{\beta_i\in \Bbbk|\;1\leq i \leq n\}$ satisfies the following conditions
\begin{itemize}
  \item[(i)] $w^1(s_i,s_j)=w^1(s_j,s_i)$;
  \item[(ii)] $\beta_i^{k_i}=P_{s_i^{k_i}},\;w^1(s_i,a^2)=\beta_i^2\sigma(s_i)$;
   \item[(iii)]$w^1(s_i,b)=\eta(a,s_i)$;
\end{itemize}
where $n=|S|$ and $1\leq i ,j \leq n$.
\end{proposition}

\begin{proof}
If $R$ is a $\varphi$-symmetric quasitriangular structure, then we can find a bicharacter $w^1$ on $S$ and a pairing $(\beta_i,\gamma_i)_{1\leq i \leq n}$ satisfy (ii), (iii) by Corollary \ref{coro5.2.1}. Since Corollary \ref{coro6.1.1}, we know $w^1$ satisfies (i). Conversely, it can be seen that $w^1$ and $\{\beta_i\in \Bbbk|\;1\leq i \leq n\}$ such that conditions of Corollary \ref{coro5.2.2}, so we can find a quasitriangular structure $R$ on $\Bbbk^G\#_{\sigma,\tau}\Bbbk \mathbb{Z}_{2}$ satisfies $w^1$ of $R$ is exactly the $w^1$ and $w^2(s_i,a)=\beta_i$. Then we will show that $R$ is $\varphi$-symmetric and hence we complete the proof. Since Corollary \ref{coro6.1.1}, we only need to prove that $w^1(s_i,s_j)=w^1(s_j,s_i),\;w^2(s_i,a)=w^3(ab,s_i)$ for $1\leq i,j\leq n$. Owing to (i), we know $w^1(s_i,s_j)=w^1(s_j,s_i)$. Because of the proof of Corollary \ref{coro5.2.2}, we get $w^3(a,s_i)=\eta(a,s_i)w^2(s_i,a)$. Due to (ii) of Lemma \ref{lem4.1.1}, we obtain $w^3(ab,s_i)=w^3(a,s_i)\eta(a,s_i)$. Therefore $w^3(ab,s_i)=\eta(a,s_i)^2w^2(s_i,a)$. But $\eta(a,s_i)^2=\eta(a^2,s_i)=1$, so  $w^3(ab,s_i)=w^2(s_i,a)$.
\end{proof}

\section{All quasitriangular structures on $K(8n,\sigma,\tau),\;A(8n,\sigma,\tau)$}
We consider $K(8n,\sigma,\tau),\;A(8n,\sigma,\tau)$ to be the simplest Hopf algebras satisfying all the conditions in Proposition \ref{pro2.1.2} due to the numbers of generators of $G$ are very small. Moreover, we will see that $\Bbbk^G\#_{\sigma,\tau}\Bbbk \mathbb{Z}_{2}$ has a quotient, either $K(8n,\sigma,\tau)$ or $A(8n,\sigma,\tau)$. For these reasons, we will give all the quasitriangular structures on $K(8n,\sigma,\tau),\;A(8n,\sigma,\tau)$ in this section. Let $\Bbbk^G\#_{\sigma,\tau}\Bbbk \mathbb{Z}_{2}$ be in Definition
\ref{def2.1.2}, and if there is a subgroup $H$ of $G$ such that $H\triangleleft x=H$, then we have another data $(H,\triangleleft|_H,\sigma|_{H},\tau|_{H\times H})$. For our convenience, we denote the data $(H,\triangleleft|_H,\sigma|_{H},\tau|_{H\times H})$ as $(H,\triangleleft,\sigma,\tau)$.

\begin{proposition}\label{pro7.1.1}
$\Bbbk^H\#_{\sigma,\tau}\Bbbk \mathbb{Z}_{2}$ is a quotient of $\Bbbk^G\#_{\sigma,\tau}\Bbbk \mathbb{Z}_{2}$.
\end{proposition}

\begin{proof}
We define a linear map $\psi:\Bbbk^G\#_{\sigma,\tau}\Bbbk \mathbb{Z}_{2}\rightarrow \Bbbk^H\#_{\sigma,\tau}\Bbbk \mathbb{Z}_{2}$ by letting
\begin{align*}
\psi(e_h):=e_h,\;\psi(e_g):=0,\;\psi(e_hx):=e_hx,\;\psi(e_gx):=0,
\end{align*}
where $h\in H,\;g\notin H$. Then it can be seen that $\psi$ is a morphism of Hopf algebras and $\psi$ is surjective. So we have completed the proof.
\end{proof}

\begin{corollary} \label{coro7.1.1}
$K(8n,\sigma,\tau)$ is a quotient of $\Bbbk^G\#_{\sigma,\tau}\Bbbk \mathbb{Z}_{2}$ or $A(8n,\sigma,\tau)$ is a quotient of $\Bbbk^G\#_{\sigma,\tau}\Bbbk \mathbb{Z}_{2}$.
\end{corollary}

\begin{proof}
Since (ii) of Proposition \ref{pro2.1.2}, there is $b\in S$ such that $b^2=1$ and $t\triangleleft x=tb$ for $t\in T$. Taking $a\in T$ and let $H:=\langle a,b\rangle$ as subgroup of $G$, then we know that $\Bbbk^H\#_{\sigma,\tau}\Bbbk \mathbb{Z}_{2}$ is a quotient of $\Bbbk^G\#_{\sigma,\tau}\Bbbk \mathbb{Z}_{2}$ by Proposition \ref{pro7.1.1}. Next we will show that $H=\langle a,b|\;a^{2n}=1,b^2=1,ab=ba\rangle$ or $H=\langle a,b|\;a^{4n}=1,b=a^{2n}\rangle$ as group for some $n\in \mathbb{N}$ and thus we complete the proof. If $b\in \langle a\rangle$, then we can assume $b=a^{m}$ for some $m\in \mathbb{N}$. Since $b^2=1$, then we have $a^{2m}=1$. we claim that $m$ is an even number in this case. Otherwise, $m$ is odd and then we have $a^m\in S$. Because $a^2\in S$ by definition and $(2,m)=1$, so we get $a\in S$. But this is a contradiction, and hence we can assume that $m=2n$. Then it can be seen that $H=\langle a,b|\;a^{4n}=1,b=a^{2n}\rangle$ as group. If $b\notin \langle a\rangle$, then we will show that $H=\langle a,b|\;a^{2n}=1,b^2=1,ab=ba\rangle$. Since $a^2\in S$ and $a\notin S$, we can assume that the order of $a$ is $2n$ for some $n\in \mathbb{N}$. Let $i,j\in \mathbb{N}$ and if $a^ib^j=1$, then we have $2|j$ due to $b\notin \langle a\rangle$. Then we know $a^i=1$ and hence $(2n)|i$. Therefore we get $H=\langle a,b|\;a^{2n}=1,b^2=1,ab=ba\rangle$.
\end{proof}

Not only $K(8n,\sigma,\tau),\;A(8n,\sigma,\tau)$ have the simplest form, but also their quasitriangular structures are very simple.

\begin{proposition}\label{pro7.1.2}
All quasitriangular structures on $K(8n,\sigma,\tau),\;A(8n,\sigma,\tau)$ are $\varphi$-symmetric.
\end{proposition}

\begin{proof}
Let $R$ be a non-trivial quasitriangular structure on $K(8n,\sigma,\tau)$, then we will show that $R$ is $\varphi$-symmetric. Owing to the definition of $K(8n,\sigma,\tau)$, we can assume that $s_1=a^2,\;s_2=b$. Since Corollary \ref{coro6.1.1}, we only need to show that the following equations hold
$$w^1(s_1,s_2)=w^1(s_2,s_1),\;w^2(s_1,a)=w^3(ab,s_1),\;w^2(s_2,a)=w^3(ab,s_2).$$
Because (iii) of Corollary \ref{coro5.2.1}, we get $w^1(s_1,b)=w^1(b,s_1)=\eta(a,s_1)$. But $s_2=b$ and so we obtain $w^1(s_1,s_2)=w^1(s_2,s_1)$. Due to \eqref{e3.14} of Lemma \ref{lem4.1.1}, we have $w^4(ab,a)=w^4(a,ab)$. Since $l(X_a)^2=\tau(a,a)l(X_{a^2})$, we get $w^4(a,t)w^4(a,t\triangleleft x)\sigma(t)=\tau(a,a)w^2(a^2,t)$ by expanding the equation. Let $t=a$, then we have
\begin{align}
\label{e7.1} w^4(a,a)w^4(a,ab)\sigma(a)=\tau(a,a)w^2(a^2,a).
\end{align}
Similarly, we obtain that $w^4(t,a)w^4(t\triangleleft x,a)\sigma(t)=\tau(a,a)w^3(t,a^2)$ by expanding $r(X_a)^2=\tau(a,a)r(X_{a^2})$. Let $t=a$, then we have
\begin{align}
\label{e7.2} w^4(a,a)w^4(ab,a)\sigma(a)=\tau(a,a)w^3(a,a^2).
\end{align}
Since $w^4(a,ab)=w^4(ab,a)$ and the equations \eqref{e7.1}, \eqref{e7.2}, we get $w^2(a^2,a)=w^3(a,a^2)$. Because of \eqref{e3.13} in Lemma \ref{lem4.1.1}, the know $w^3(ab,a^2)=w^3(a,a^2)$ and so the equation $w^2(s_1,a)=w^3(ab,s_1)$ holds. To show $w^2(s_2,a)=w^3(ab,s_2)$, we use (ii) of Lemma \ref{lem4.1.2} and we get $w^2(b,a)=\tau(b,a)\frac{w^4(ab,a)}{w^4(a,a)}$. Similarly, we get $w^3(ab,b)=\tau(b,a)\frac{w^4(a,ab)}{w^4(a,a)}$ by (iii) of Lemma \ref{lem4.1.2}. Because we have known $w^4(a,ab)=w^4(ab,a)$, we get $w^2(b,a)=w^3(ab,b)$ and so $w^2(s_2,a)=w^3(ab,s_2)$. Therefore $R$ is $\varphi$-symmetric. Similarly, one can prove that all quasitriangular structures on $A(8n,\sigma,\tau)$ are $\varphi$-symmetric.
\end{proof}

Let $Q_K:=\{$non-trivial quasitriangular structures on $K(8n,\sigma,\tau)$ $\}$, then we have
\begin{theorem}\label{thm7.1.1}
$Q_K \stackrel{1-1}{\longleftrightarrow}\{(\beta_1,\beta_2,\delta)|\;\beta_1^n=P_{s_1^{n}},\;\beta_2^2=P_{s_2^2},\;\delta^2=\frac{\tau(a,a)\tau(b,a)}{\tau(b,b)\sigma(a)}\beta_1\beta_2\}$, where $s_1=a^2,\;s_2=b$.
\end{theorem}

\begin{proof}
Given a non-trivial quasitriangular structure $R$ on $K(8n,\sigma,\tau)$, we can define a triple $(\beta_1,\beta_2,\delta)$ through letting $\beta_1:=w^2(s_1,a),\;\beta_2:=w^2(s_2,a),\;\delta:=w^4(a,a)$. Since (ii), (iv) of Proposition \ref{pro5.2.1}, we know $\beta_1^n=P_{s_1^{n}},\;\beta_2^2=P_{s_2^2},\;\delta^2=\frac{\tau(a,a)\tau(b,a)}{\tau(b,b)\sigma(a)}\beta_1\beta_2$. Conversely, let $(\beta_1,\beta_2,\delta)$ be a triple satisfying $\beta_1^n=P_{s_1^{n}},\;\beta_2^2=P_{s_2^2},\;\delta^2=\frac{\tau(a,a)\tau(b,a)}{\tau(b,b)\sigma(a)}\beta_1\beta_2$, then we claim that there is a unique quasitriangular structure $R$ such that $w^2(s_1,a)=\beta_1,\;w^2(s_2,a)=\beta_2,\;w^4(a,a)=\delta$. To do this, let $w^1$ be a bicharacter on $S$ which is determined as follows
\begin{align}
\label{e7.4} w^1(s_1,s_1):=\beta_1^2\sigma(s_1),\;w^1(s_1,s_2)=w^1(s_2,s_1):=1,\;w^1(s_2,s_2):=\eta(a,s_2).
\end{align}
then we will use Proposition \ref{pro6.1.2} to get a quasitriangular structure $R$ such that $w^2(s_1,a)=\beta_1,\;w^2(s_2,a)=\beta_2,\;w^4(a,a)=\delta$. We first show $w^1$ is well defined. To show this, the only non-trivial thing is to prove that $[\beta_1^2\sigma(s_1)]^n=1$. Since Lemma \ref{lem5.2.1}, we get $P_{s_1^n}^2\sigma(s_1)^n=1$ and so $[\beta_1^2\sigma(s_1)]^n=1$. Then we prove that $w^1$ and the set $\{\beta_1,\beta_2\}$ such that the conditions of Proposition \ref{pro6.1.2}. To prove this, the only non-trivial thing is to show $\beta_2^2\sigma(s_2)=1$. Owing to $\tau(a,b)\tau(ab,b)=\sigma(ab)\sigma(a)^{-1}\sigma(b)^{-1}$ and $\sigma(ab)=\sigma(a)$, we know $\tau(b,b)\sigma(b)=1$. Due to $\beta_2^2\sigma(s_2)=\tau(b,b)\sigma(b)$, we have $\beta_2^2\sigma(s_2)=1$. Now we can use Proposition \ref{pro6.1.2} and Proposition \ref{pro5.2.4} to get a $\varphi$-symmetric quasitriangular structure $R$ satisfying $w^2(s_1,a)=\beta_1,\;w^2(s_2,a)=\beta_2,\;w^4(a,a)=\delta$. Since Lemma \ref{lem4.1.2}, we know $R$ is unique if it is a $\varphi$-symmetric quasitriangular structure and it satifies that $w^2(s_1,a)=\beta_1,\;w^2(s_2,a)=\beta_2,\;w^4(a,a)=\delta$. Finally, since Proposition \ref{pro7.1.2}, we know that this correspondence we have discussed is one-to-one.
\end{proof}

\begin{remark}\emph{In fact, from the proof of the above theorem, we know that all non-trivial quasitriangular structures on $K(8n,\sigma,\tau)$ are given by (i)-(iv) of Proposition \ref{pro5.2.4}, where $(\beta_1,\beta_2,\delta)$ are in Theorem \ref{thm7.1.1} and $w^1$ is defined by \eqref{e7.4} above and $\alpha_{ij}=w^1(s_i,s_j),\;\gamma_i=\beta_i \eta(a,s_i)$ for $1\leq i\leq 2$.}\end{remark}

Similar to above Theorem \ref{thm7.1.1}, let $Q_A:=\{$non-trivial quasitriangular structures on $A(8n,\sigma,\tau)$ $\}$, then we have
\begin{theorem}\label{thm7.1.2}
$Q_A \stackrel{1-1}{\longleftrightarrow}\{(\beta,\delta)|\;\beta^{2n}=P_{s^{2n}},\;\delta^2=\frac{\tau(a,a)\tau(b,a)}{\tau(b,b)\sigma(a)}P_{s^n}\beta^{1+n}\}$, where $s=a^2$.
\end{theorem}

\begin{proof}
Let $R$ be a non-trivial quasitriangular structure on $K(8n,\sigma,\tau)$ and let $w^1(s,s)=\beta,w^4(a,a)=\delta$, then it can be seen that $(\beta,\delta)$ such that the following conditions
\begin{align}
\label{e7.3} \beta^{2n}=P_{s^{2n}},\;\delta^2=\frac{\tau(a,a)\tau(b,a)}{\tau(b,b)\sigma(a)}P_{s^n}\beta^{1+n}.
\end{align}
due to (ii),(iv) of Proposition \ref{pro5.2.1}. Conversely, if $(\beta,\delta)$ satisfies the equation \eqref{e7.3}, then we can use (i)-(iv) of Proposition \ref{pro5.2.4} to define a $R$ as follows
\begin{align}
\label{e7.5} w^1(s,s):=\beta^2\sigma(s),\;w^2(s,a)=w^3(a,s):=\beta,\;w^4(a,a):=\delta.
\end{align}
Then we can see that the four tuple $(\beta^2\sigma(s),\beta,\beta,\delta)$ satisfies conditions (i)-(vi) of Proposition \ref{pro5.2.1} and thus $R$ is a non-trivial quasitriangular structure. Moreover, since Lemma \ref{lem4.1.2}, we know $R$ is unique if it is a $\varphi$-symmetric quasitriangular structure and satisfies $w^2(s,a)=w^3(s,a)=\beta,\;w^4(a,a)=\delta$. Finally, since Proposition \ref{pro7.1.2}, we know that this correspondence we have discussed is one-to-one.
\end{proof}

\begin{remark}\emph{From the proof of the above theorem, we know that all non-trivial quasitriangular structures on $A(8n,\sigma,\tau)$ are given by (i)-(iv) of Proposition \ref{pro5.2.4}, where $(\beta,\delta)$ are in Theorem \ref{thm7.1.2} and $w^1$ is defined by \eqref{e7.5} above and $\alpha_{11}=w^1(s,s),\;\gamma=\beta $.}\end{remark}

\end{document}